\documentclass[11pt]{amsart}

\usepackage[pdftex]{graphicx}
\usepackage{amssymb}

\usepackage{amsmath}
\usepackage{amsfonts}
\usepackage{amsthm}
\usepackage{amssymb}
\usepackage{xcolor}

\numberwithin{equation}{section}
\newtheorem{theorem}{Theorem}[section]

\newtheorem{proposition}[theorem]{Proposition}

\newtheorem{claim}[theorem]{Claim}

\newtheorem{lemma}[theorem]{Lemma}

\theoremstyle{definition}
\newtheorem{definition}[theorem]{Definition}
\theoremstyle{remark}
\newtheorem{remark}[theorem]{Remark}

\newcommand{\ind}{\mathrm{ind}}
\newcommand{\nul}{\mathrm{nul}}

\newcommand{\area}{\mathrm{Area}}
\newcommand{\length}{\mathrm{Length}}

\begin{document}
\title[Min-max characterization of the conformal eigenvalues]{Min-max harmonic maps and a new characterization of conformal eigenvalues
}
\begin{abstract}
Given a surface $M$ and a fixed conformal class $c$ one defines $\Lambda_k(M,c)$ to be the supremum of the $k$-th nontrivial Laplacian eigenvalue over all metrics $g\in c$ of unit volume. It has been observed by Nadirashvili that the metrics achieving $\Lambda_k(M,c)$ are closely related to harmonic maps to spheres. 
In the present paper, we identify $\Lambda_1(M,c)$ and $\Lambda_2(M,c)$ with min-max quantities associated to the energy functional for sphere-valued maps. 
As an application, we obtain several new eigenvalue bounds, including a sharp isoperimetric inequality for the first two Steklov eigenvalues. This characterization also yields an alternative proof of the existence of maximal metrics realizing $\Lambda_1(M,c)$, $\Lambda_2(M,c)$ and, moreover, allows us to obtain
a regularity theorem for maximal Radon measures satisfying a natural compactness condition.

\end{abstract}

\author[M. Karpukhin]{Mikhail Karpukhin}
\address{Department of Mathematics\\
University College London\\
25 Gordon Street\\
London, WC1H 0AY, UK
}
\email{m.karpukhin@ucl.ac.uk}
\author[D. Stern]{Daniel Stern}
\address{
Department of Mathematics\\
University of Chicago\\
5734 S. University Ave,
Chicago, IL 60637, USA
}
\email{dstern@uchicago.edu}
\maketitle

\section{Introduction}

\subsection{Eigenvalues of the Laplacian}
Let $(M,g)$ be a closed Riemannian surface, and let $\Delta_g\colon C^\infty(M)\to C^\infty(M)$ be the associated Laplace operator with positive spectrum
$$
0=\lambda_0(M,g)<\lambda_1(M,g)\leqslant \lambda_2(M,g)\leqslant\ldots\nearrow\infty,
$$
where eigenvalues are written with multiplicities. Multiplying eigenvalues by the area $\area(M,g)$, one obtains the scale-invariant quantity
$$
\bar\lambda_k(M,g) = \lambda_k(M,g)\area(M,g).
$$
By results of Yang-Yau~\cite{YY} for $k=1$ and Korevaar~\cite{Korevaar} for $k\geq 1$ there exists a constant $C(M)$ depending only on the topology of $M$ such that for any metric $g$ one has $\bar\lambda_k(M,g)\leqslant C(M)k$.
Given a conformal class $c=[g]=\{g|\,g=f g_0,\,f> 0\}$ on $M$, it therefore makes sense to consider the conformal supremum of $\bar{\lambda}_k$, denoted by
$$
\Lambda_k(M,c) = \sup\limits_{g\in c}\bar\lambda_k(M,g).
$$
The quantities $\Lambda_k(M,c)$ are often referred to as the {\em conformal spectrum} of $(M,c)$, see~\cite{CES}.
Interest in the quantity $\Lambda_k(M,[g])$ stems in large part from the connection between extremal metrics for $\bar{\lambda}_k$ and the theory of harmonic maps to spheres, as described in the following theorem (see Section~\ref{intro_harm:sec} for more details). 

\begin{theorem}[Nadirashvili~\cite{NadirashviliTorus}, El Soufi-Ilias~\cite{ESI}, see also~\cite{FS2}]
\label{extremalThm:intro}
Let $h\in c=[g]$ be a metric such that 
\begin{equation}
\label{optimizer:eq}
\Lambda_k(M,c) = \bar\lambda_k(M,h).
\end{equation}
Then there exists a harmonic map $\Phi\colon(M,g)\to\mathbb{S}^n$ 
such that the components of $\Phi$ are $\lambda_k(M,h)$-eigenfunctions.
\end{theorem}

As a result, the problem of exhibiting a metric $g\in c$ satisfying~\eqref{optimizer:eq} is of interest not only from the perspective of spectral theory, but also as a means for producing a distinguished collection of harmonic maps from Riemann surfaces into spheres. Our understanding of this problem has seen significant progress in the recent years: we refer the reader to~\cite{Petrides, Petridesk, NS, KNPP2}, where two different approaches to this problem are developed; see also Section~\ref{intro_harm:sec} below for details. 

The connection between $\bar{\lambda}_k$-extremal measures and sphere-valued harmonic maps hints at the possibility of a deeper relationship between the conformal spectrum and the variational theory of the Dirichlet energy for sphere-valued maps. In the present paper, we make this relationship explicit in the cases $k=1,2$, characterizing $\Lambda_1(M,c)$ and $\Lambda_2(M,c)$ as the min-max energies associated to certain families of sphere-valued maps on $M$.
%
%
%
%

\subsection{Min-max characterization of $\Lambda_1(M,c)$ and applications.}

Let $(M,g)$ be a closed Riemannian surface. For the purposes of intuition, we introduce the collection $\widetilde{\Gamma}_n(M)$ of families
$$\overline{B}^{n+1}\ni a\mapsto F_a\in W^{1,2}(M,\mathbb{S}^n)\text{ such that }F_a\equiv a\text{ for }a\in \mathbb{S}^n$$
continuous with respect to the \emph{weak topology} on $W^{1,2}$. A motivating example comes from composing a given map $M\to\mathbb{S}^n$ with the family of conformal dilations of $\mathbb{S}^n$, as described in Section 2.4 below.

By standard topological arguments, the boundary conditions imposed on $F\in \widetilde{\Gamma}_n(M)$ force the existence of a map $F_y$ in the family with zero average $\int_M F_y=0\in \mathbb{R}^{n+1}$, so that the Dirichlet energies $E(F_y)$ satisfy
$$
\sup_a 2E(F_a)\geqslant \lambda_1(M,g)\|F_y\|_{L^2(M,g)}^2=\bar{\lambda}_1(M,g).
$$ 
In particular, since the Dirichlet energy is a conformal invariant, it follows that the maximal eigenvalue $\Lambda_1(M,[g])$ is bounded above by an associated {\em min-max energy}
$$
\widetilde{\mathcal E_n} = \inf_{F\in \widetilde\Gamma_n(M)}\sup_{a\in B^{n+1}} E(F_a)\geqslant \frac{1}{2}\Lambda_1(M,[g]),
$$
similar to the classical conformal volume bounds of Li and Yau \cite{LiYau}. For technical reasons clarified below, we do not work directly with $\widetilde{\mathcal E_n}$, but introduce a related min-max energy $\mathcal E_n\leqslant \widetilde{\mathcal E_n}$, defined via a relaxation of Ginzburg-Landau type, see~\eqref{GL:example} below. For $n>5$, our first result confirms that these min-max energies $\mathcal{E}_n$ are achieved as the energies of harmonic maps to $\mathbb{S}^n$. Crucially, by virtue of the min-max construction, these maps also come with natural bounds on their \emph{energy index}--i.e., their Morse index as critical points of the energy functional.
\begin{theorem}
\label{MainTh1:intro}
Let $n>5$. Then for any Riemannian surface $(M,g)$, there exists a harmonic map $\Psi_n:M\to \mathbb{S}^n$ such that 
\begin{equation}
\label{Ebound:eq}
\mathcal E_n = E(\Psi_n)\geqslant \frac{1}{2}\Lambda_1(M,c),
\end{equation}
whose energy index $\ind_E(\Psi_n)$ satisfies
$$
\ind_E(\Psi_n)\leqslant n+1.
$$
\end{theorem}

Note that the right hand side of~\eqref{Ebound:eq} does not depend on $n$. Therefore, it makes sense to study the behavior of this inequality as $n$ becomes large. Our second result is the following, showing that \eqref{Ebound:eq} becomes an equality for $n$ sufficiently large.

\begin{theorem}[Min-max characterization of $\Lambda_1(M,c)$]
\label{MainTh2:intro}
Given a surface $M$ and a conformal class $c$ on $M$, there exists $N=N(M,c)$ such that for all $n\geqslant N$, the components of $\Psi_n$ lie in the first positive eigenspace of the Laplacian $\Delta_{g_{\Psi_n}}$ for the conformal metric $g_{\Psi_n}=|d\Psi_n|_g^2g$ (which may have conical singularities). In particular, 
$$
\mathcal E_n = E(\Psi_n) = \frac{1}{2}\Lambda_1(M,c).
$$
\end{theorem}

As an immediate consequence, one sees that our min-max procedure provides an alternative construction of conformally maximizing metrics for $\bar{\lambda}_1(M,g)$. While the existence of maximizing metrics has been established in~\cite{Petrides, KNPP2} by other methods, the novel feature of Theorem \ref{MainTh2:intro} is the identification of the supremal eigenvalue $\frac{1}{2}\Lambda_1(M,[g])$ with the min-max energies $\mathcal{E}_n$ for $n$ sufficiently large. This characterization leads to a number of new estimates relating $\Lambda_1(M,[g])$ to other spectral quantities, allowing us to refine many known eigenvalue bounds involving the Li--Yau conformal volume $V_c(M,[g])$, by replacing $V_c(M,[g])$ with $\frac{1}{2}\Lambda_1(M,[g])$. In several cases of interest--as we will see below--these refined estimates in terms of $\Lambda_1(M,[g])$ turn out to be sharp.

%
%

Let us describe some of the applications of this min-max characterization. 
In~\cite{Kokarev}, Kokarev defined a natural analog of Laplacian eigenvalues $\lambda_k(M,c,\mu)$ associated to any {\em Radon measure} $\mu$ on a surface $M$ endowed with a conformal class $c$, and noted that the first normalized eigenvalue $\bar\lambda_1(M,c,\mu)$ is bounded from above by twice the conformal volume. We are able to replace the conformal volume by $\mathcal E_n$ in this estimate, provided that the Radon measure satisfies a certain natural compactness condition (see Theorem~\ref{intro_regularity:thm}). We call such measures {\em admissible}. As a result, one has that $\Lambda_1(M,c)$ is an upper bound for the first normalized eigenvalue of any admissible Radon measure. Moreover, we are able to characterize the equality case, arriving at a {\em regularity theorem} for $\bar\lambda_1$-maximal admissible measures. This answers Question 1 in~\cite[Section 6.2]{Kokarev}.

\begin{theorem}[Regularity theorem for $\bar\lambda_1$-maximal measures]
\label{intro_regularity:thm}
Let $\mu$ be a Radon measure on $(M,g)$ such that the map
$$
T\colon W^{1,2}(M,g)\to L^2(M,\mu)
$$
is well-defined and compact. Then one has 
\begin{equation}
\label{intro_measures:ineq}
\bar\lambda_1(M,[g],\mu)\leqslant \Lambda_1(M,[g]).
\end{equation}
Suppose that $\mu$ is $\bar\lambda_1$-maximal, i.e. that inequality~\eqref{intro_measures:ineq} is an equality. Then there exists a harmonic map $\Phi\colon (M,g)\to\mathbb{S}^n$ such that
$$
d\mu = \frac{1}{\lambda_1(M,[g],\mu)}|d\Phi|_g^2\,dv_g = \frac{1}{\lambda_1(M,[g],\mu)}\,dv_{g_\Phi},
$$
where $g_\Phi = |d\Phi|_g^2\,dv_g$,  and the components of $\Phi$ are the first eigenfunctions of the Laplacian $\Delta_{g_{\Phi}}$. In particular, $d\mu$ is smooth. 
\end{theorem}

\begin{remark}
A year after the present paper was posted, in joint work with M.~Nahon and I. Polterovich \cite{KNPS}, we used the min-max characterization to establish the \emph{stability} of metrics maximizing $\bar{\lambda}_1$ in a conformal class: we show that any sequence of admissible measures $\mu_j$ with $\bar{\lambda}_1(M,[g],\mu_j)\to \Lambda_1(M,[g])$ subsequentially converges in $(W^{1,2})^*$ to a (smooth) conformally $\bar{\lambda}_1$-maximizing measure.
\end{remark}

Kokarev used his observation to obtain an upper bound for the first normalized Steklov eigenvalue on surfaces with boundary, independent of the number of boundary components of the surface~\cite[Theorem $A_1$]{Kokarev}. Recall that for a domain $(\Omega,g)\subset (M,g)$, the \emph{Steklov eigenvalues} $\sigma_k(\Omega,g)$ are defined to be the eigenvalues of the Dirichlet-to-Neumann operator $\mathcal D_g$ on $\partial \Omega$, whose spectrum is discrete if $\Omega$ is e.g. Lipschitz, see~\cite{GP,FS,FS2} for surveys of recent results. The theory of maximal metrics for Steklov eigenvalues has strong parallels with that of Laplacian eigenvalues on closed surfaces, as discussed in Section~\ref{steklov.sec} below.

As a corollary of Theorem~\ref{intro_regularity:thm} we obtain the following. 

\begin{theorem}
\label{Steklovthm:intro}
Let $\Omega\subset M$ be a Lipschitz domain. Then one has
\begin{equation}
\label{Steklov:ineq}
\sigma_1(\Omega,g)\length(\partial \Omega,g)< \Lambda_1(M,[g]).
\end{equation}
\end{theorem}
\begin{remark}
In a recent preprint~\cite{GLS}, the authors use homogenisation techniques to show that by making many small holes in $M$ one can find a sequence of domains $\Omega_n\subset M$ such that 
$$
\lim_{n\to\infty}\sigma_k(\Omega_n,g)\length(\partial \Omega_n,g) = \bar\lambda_k(M,g).
$$
This means that inequality~\eqref{Steklov:ineq} is in fact sharp.
\end{remark}
\begin{remark}
In fact, Theorem~\ref{Steklovthm:intro} holds under much weaker assumptions on $\Omega$--namely, we show that \eqref{Steklov:ineq} holds whenever the trace map $W^{1,2}(M,g)\to L^2(\partial \Omega)$ is compact.
\end{remark}
\begin{remark}
Several months after the present paper was posted, a non-strict version of~\eqref{Steklov:ineq} was reproved in~\cite{GKL} using a direct approximation argument. In particular, it was observed in~\cite{GKL} that the same argument yields an analogous inequality for any index of the eigenvalue, i.e. one has
$$
\sigma_k(\Omega,g)\length(\partial \Omega,g)\leqslant\Lambda_k(M,[g])
$$
for all $k\geqslant 1$.\end{remark}


In Section~\ref{Steklov_applications:sec} we discuss some applications of Theorem~\ref{Steklovthm:intro} and results of~\cite{GLS} to optimization of Steklov eigenvalues. In particular, we obtain the following result (see~\cite{FS} for related results).

\begin{theorem}
Let $\Omega_{\gamma,b}$ be an orientable surface of genus $\gamma$ with $b$ boundary components. Define
$$
\Sigma_1(\gamma,b) = \sup_g\sigma_1(\Omega_{\gamma,b},g)\length(\partial \Omega_{\gamma,b},g).
$$
Then there are infinitely many $\gamma\geqslant 0$ such that for each such $\gamma$ there are infinitely many $b\geqslant 1$ for which the quantity
$\Sigma_1(\gamma,b)$ is achieved by a smooth metric. In particular, for such $(\gamma,b)$ there exists a free boundary minimal branched immersion $f\colon \Omega_{\gamma,b}\to B^{n_{\gamma,b}}$ by the first Steklov eigenfunctions.
\end{theorem}

\begin{remark}
We remark that Theorem \ref{Steklovthm:intro}, together with the results of \cite{GLS}, provide precise asymptotic description of the areas $\frac{1}{2}\Sigma_1(\gamma,b)$ of the associated free boundary minimal surfaces as $b\to\infty$, showing that they approach the supremum $\Lambda_1(\gamma)$ of $\bar{\lambda}_1(M,g)$ over all metrics on the closed surface of genus $\gamma$. Indeed, roughly a year after the appearance of the present paper and \cite{GLS}, we proved in \cite{KS21} that the free boundary minimal surfaces in the Euclidean ball realizing $\Sigma_1(\gamma,b)$ converge (e.g., as varifolds) to closed minimal surfaces in the sphere realizing $\Lambda_1(\gamma)$ as $b\to\infty$, with areas converging at the rate $\Lambda_1(\gamma)-\Sigma_1(\gamma,b)\sim \frac{\log b}{b}$.
\end{remark}

%

\subsection{Min-max characterization of $\Lambda_2(M,c)$ and applications.} Using similar techniques, we are also able to give a min-max characterization of the maximal second eigenvalue $\Lambda_2(M,c)$. Inspired by Nadirashvili's computation of $\Lambda_2(\mathbb{S}^2)$~\cite{NadirashviliS2} and the subsequent works~\cite{PetridesS2, GNP, GL}, we introduce a $2(n+1)$-parameter min-max construction for harmonic maps to $\mathbb{S}^n$, whose associated min-max energy $\mathcal{E}_{n,2}$ satisfies
\begin{equation}\label{en2.ineq}
{\mathcal E_{n,2}}\geqslant \frac{1}{2}\Lambda_2(M,c).
\end{equation}
This energy is achieved by a harmonic map to $\mathbb{S}^n$, possibly together with a bubble, as described in the following theorem.

\begin{theorem}
\label{MainTh3:intro}
Let $n\geqslant 9$. Then one of the following two situations occurs
\begin{itemize}
\item[(a)] There exists a harmonic map $\Phi_{n,2}$ such that $\mathcal E_{n,2} = E(\Phi_{n,2})$ and $\ind_E(\Phi_{n,2})\leqslant 2(n+1)$.
\item[(b)] There exists a harmonic map $\Phi_{n,2}$ such that $\mathcal E_{n,2} = E(\Phi_{n,2})+4\pi$ and  
$\ind_E(\Phi_{n,2})\leqslant n+4$.
\end{itemize}
\end{theorem}
Moreover, for $n$ sufficiently large, we show that equality holds in \eqref{en2.ineq}.

\begin{theorem}[Min-max characterization of $\Lambda_2(M,c)$]
\label{MainTh4:intro}
Given $(M,c)$ there exists $N=N(M,c)$ such that for all $n\geqslant N$ one has
$$
\mathcal E_{n,2} =\frac{1}{2}\Lambda_2(M,c).
$$
\end{theorem}
As an application, one obtains a new proof of the existence of $\bar\lambda_2$-maximal metrics and the analogs of Theorems~\ref{intro_regularity:thm} and~\ref{Steklovthm:intro} for $k=2$. 

\begin{theorem}[Regularity theorem for $\bar\lambda_2$-maximal measures]
\label{intro_regularity2:thm}
Let $\mu$ be a Radon measure on $(M,g)$ such that the map
$$
T\colon W^{1,2}(M,g)\to L^2(\mu)
$$
is well-defined and compact. 
Then one has 
\begin{equation}
\label{intro_measures2:ineq}
\bar\lambda_2(M,[g],\mu)\leqslant \Lambda_2(M,[g]).
\end{equation}
Suppose that $\mu$ is $\bar\lambda_2$-maximal, i.e. that inequality~\eqref{intro_measures2:ineq} is an equality. Then there exists a harmonic map $\Phi\colon (M,g)\to\mathbb{S}^n$ such that
$$
d\mu = \frac{1}{\lambda_2(M,[g],\mu)}|d\Phi|_g^2\,dv_g = \frac{1}{\lambda_2(M,[g],\mu)}\,dv_{g_\Phi},
$$
where $g_\Phi = |d\Phi|_g^2\,dv_g$,  and the components of $\Phi$ are the second eigenfunctions of the Laplacian $\Delta_{g_{\Phi}}$. In particular, $d\mu$ is smooth.
\end{theorem}

\begin{remark}
In the subsequent paper \cite{KNPS} with M.~Nahon and I. Polterovich, we were also able to use the min-max characterization to obtain stability results for metrics maximizing the second eigenvalue $\bar{\lambda}_2$.
\end{remark}

\begin{theorem}
\label{Steklovthm2:intro}
Let $\Omega\subset M$ be a Lipschitz domain. Then one has
\begin{equation}
\label{Steklov2:ineq}
\sigma_2(\Omega,g)\length(\partial \Omega,g)< \Lambda_2(M,[g]).
\end{equation}
\end{theorem}
\begin{remark}
Once again the results of~\cite{GLS} imply that the inequality~\eqref{Steklov2:ineq} is sharp.
\end{remark}

\subsection{Ideas of the proofs}

Theorem~\ref{MainTh1:intro} is proved using variational techniques. Rather than applying variational methods directly to the Dirichlet energy on the space $W^{1,2}(M,\mathbb{S}^n)$, we introduce a min-max procedure for a family of relaxed functionals $E_{\epsilon}$ of Ginzburg-Landau type on the space $W^{1,2}(M,\mathbb{R}^{n+1})$ (formed by combining the Dirichlet energy with a nonlinear potential penalizing deviation from $\mathbb{S}^n\subset \mathbb{R}^{n+1}$). Since these perturbed functionals are $C^2$ functions on the Hilbert space $W^{1,2}(M,\mathbb{R}^{n+1})$ satisfying a Palais-Smale condition, it is easy to produce critical points via standard min-max methods, which (by the results of \cite{LW}) converge as $\epsilon\to 0$ to a harmonic $\mathbb{S}^n$-valued map, possibly with some bubbles. Moreover, for the maps achieving the first min-max energy $\mathcal{E}_n(M)$, the sum of energy indices of the map and the bubbles is at most $n+1$. We then use index bounds~\cite{KarRP2} in conjunction with the eigenvalue rigidity estimate of Petrides~\cite{Petrides} to show that, in this case, there are no bubbles.

Theorem~\ref{MainTh3:intro} is proved in essentially the same way. The only difference is that, in the last step, one can not rule out the possibility that $\mathcal E_{n,2}$ is achieved by a harmonic map together with a single totally geodesic bubble. Note that such bubbles can indeed occur for a $\bar\lambda_2$-maximal metric (see e.g.~\cite{PetridesS2, KNPP}) so one can not expect to rule out bubbling behavior for maps realizing $\mathcal E_{n,2}$.

Theorems~\ref{MainTh2:intro} and~\ref{MainTh4:intro} are proved using the following proposition, which could be of independent interest. We say that the map $\Psi_n\colon M\to\mathbb{S}^n$ is \emph{linearly full} if its image linearly spans $\mathbb{R}^{n+1}$.
\begin{proposition}
\label{MainProp:intro}
For any closed surface $(M,g)$ and any $E_0<\infty$, there exists an integer $N=N([g],E_0)\in \mathbb{N}$ such that if $\Psi:M\to \mathbb{S}^n$ is a linearly full harmonic map with $E(\Psi)\leqslant E_0$, then $n\leqslant N$.
\end{proposition}

In particular, given a family of harmonic maps $\Psi_n: M\to \mathbb{S}^n$ into spheres of increasing dimension satisfying a uniform energy bound, Proposition~\ref{MainProp:intro} tells us that $\Psi_n$ must take values in a totally geodesic subsphere of $\mathbb{S}^n$ for $n$ sufficiently large. The proposition is proved using a variation on the bubble convergence argument for harmonic maps. Namely, we show that (along a subsequence) the Schr\"odinger operators $\Delta_g-|d\Psi_n|_g^2$ associated to $\Psi_n$ converge in some sense to an operator with discrete spectrum. In particular, if the space of coordinate functions $\langle \Psi_n,v\rangle$ ($v\in \mathbb{R}^{n+1}$) were of unbounded dimension, then the limiting operator would have an eigenvalue of infinite multiplicity, which would contradict the discreteness of the spectrum.

\subsection{Discussion}
Recall that in~\cite{Petrides, Petridesk, KNPP2} the authors prove the existence of a metric realizing $\Lambda_k(M,c)$. In both proofs a sequence of metrics $g_m$ such that $\bar\lambda_k(M,g_m)\to\Lambda_k(M,c)$ is carefully chosen, and the convergence of the metrics $g_m$ as $m\to\infty$ is studied. The key tool is the lower bound on $\bar\lambda_k(M,g_m)$, which gives control on how the sequence $g_m$ can degenerate. The min-max characterization provides a different approach, where the sequence of metrics is replaced by a sequence of harmonic maps $\Psi_n$, and the lower bound on the eigenvalue is replaced by the upper bound on the energy index $\ind_E(\Psi_n)$; i.e. we use the index bound to control possible degenerations of the sequence $\Psi_n$. In the context of optimal eigenvalue inequalities, the index bounds were first used in~\cite{FS}. The method was further developed by the first author in~\cite{KarRP2}, where the index bounds are used to compute $\Lambda_k(\mathbb{RP}^2)$ for all $k$. The guiding principle behind~\cite{KarRP2} and the present article is that the problem of optimal eigenvalue inequalities is essentially equivalent to the problem of sharp upper bounds for the energy index of harmonic maps. Indeed, the results of the present article show how index bounds lead to the existence of maximisers for optimal eigenvalue inequalities (with the index bound $\ind_E(\Psi_n)\leq n+1\ll2n$ forcing the coordinate functions of $\Psi_n$ to be \emph{first} eigenfunctions for a suitable Laplacian), whereas in~\cite{KarRP2} this existence is combined with (almost) sharp energy index bounds in order to characterize the exact maximisers. 

A natural question is whether the min-max characterization can be proved for $\Lambda_k(M,c)$ with $k>2$. The answer is yes, provided one can produce a reasonably natural (non-empty) collection $\widetilde\Gamma_{n,k}$ of weakly continuous $k\cdot(n+1)$-dimensional families $X^{(n+1)k}\ni \alpha\mapsto F_{\alpha}\in W^{1,2}(M,\mathbb{S}^{n})$ such that
$$
\sup_{\alpha}E(F_{\alpha})\geqslant \frac{1}{2}\Lambda_k(M,c)
$$
for any $F\in \widetilde \Gamma_{n,k}$. Having that, the rest of the argument leading to the min-max characterization carries over without significant changes. 
The applications would follow immediately, including the regularity theorem for measures realizing $\Lambda_k(M,c)$ and the analogs of Theorems~\ref{Steklovthm:intro},~\ref{Steklovthm2:intro} for $k>2$. 

\begin{remark} In practice, answering this question is equivalent to the problem of finding a natural ``nonlinear energy spectrum" for the Ginzburg-Landau functionals $E_{\epsilon}:W^{1,2}(M^2,\mathbb{R}^{n+1})\to \mathbb{R}$ for $n\geq 3$. Note that in the scalar-valued case $n=0$, there is a natural definition of nonlinear energy spectrum arising from the $\mathbb{Z}_2$ symmetry of the functionals, which has recently been studied in detail by Gaspar and Guaraco in \cite{GaspGu1, GaspGu2} in connection with the volume spectrum for minimal hypersurfaces. 
\end{remark}

One should also note that any explicit construction of elements in $\widetilde \Gamma_{n,k}$ yields explicit upper bounds for $\Lambda_k(M,c)$, analogous in some sense to the classical Li--Yau bound $\Lambda_1(M,[g])\leq 2V_c(M,[g])$ for $\Lambda_1$ by the conformal volume. For example, in the course of proving Theorem~\ref{MainTh3:intro}, we obtain the following upper bound for $\Lambda_2(M,c)$.
\begin{proposition}
For any conformal class $[g]$ on any surface $M$ one has
\begin{equation}
\label{L2Vc:ineq}
\Lambda_2(M,[g])\leqslant 4V_c(M,[g]),
\end{equation}
where $V_c(M,[g])$ is the conformal volume of $(M,[g])$.
\end{proposition}
\begin{remark}
The fact that the higher conformal eigenvalues $\Lambda_k(M[g])$ are bounded purely in terms of the conformal volume was proved in~\cite{KokarevVc}, but the constants in~\cite{KokarevVc} are not explicit. At the same time, combining~\eqref{L2Vc:ineq} with the bounds for the conformal volume in~\cite{LiYau,KarNonOrientable,KokarevVc} yields an explicit bound for $\Lambda_2(M,[g])$ in terms of the topology of $M$. Similar bounds for $\Lambda_k(M,[g])$ were recently proved in~\cite[Theorem 1.6]{KNPP2}.
\end{remark}

Finally, it is worth mentioning that our take on Nadirashvili's construction~\cite{NadirashviliS2} for $\Lambda_2$ differs from the ones in~\cite{PetridesS2, GNP, GL}. We combine ideas from all four papers and present a version of the argument which appears to be simpler and completely avoids the issue of uniqueness of a renormalizing point (see~\cite{L} for some recent results on renormalization). Note that this issue recently came up in~\cite{GL}, where the authors extended Nadirashvili's construction to the Robin problem. The uniqueness of the renormalizing point turned out to be a rather delicate issue in that context, and the authors were not able to complete the proof for a certain range of Robin parameter. We believe that our version of the argument allows one to extend the range of Robin parameters for which the results of~\cite{GL} hold, see the discussion after Theorem 1 in~\cite{GL}.

\subsection{Plan of the paper} Section~\ref{prelim:sec} contains some preliminary material on eigenvalues of the Laplacian and their connection to harmonic maps. 

Section~\ref{first.min-max} is devoted to the proof of Theorem~\ref{MainTh2:intro}, the min-max characterization of $\Lambda_1(M,c)$. In Sections~\ref{first.min-max:def} we define the min-max energies $\mathcal E_n$. Theorem~\ref{MainTh1:intro} is proved in Section~\ref{first.min-max:properties}. We then prove Theorem~\ref{MainTh2:intro} using Proposition~\ref{MainProp:intro}. Finally, Section~\ref{stabilization_proof:sec} contains the proof of the most technical result of the paper, Proposition~\ref{MainProp:intro}.

In Section~\ref{second.min-max} we follow the same steps in order to show min-max characterisation of $\Lambda_2(M,c)$. In particular, Theorems~\ref{MainTh3:intro} and~\ref{MainTh4:intro} are proved in Section~\ref{second.min-max:properties} and~\ref{second.min-max:stabilization} respectively. 

Section~\ref{app.sec} contains various applications of the min-max characterization, including Theorems~\ref{Steklovthm:intro},~\ref{intro_regularity2:thm},~\ref{Steklovthm2:intro} and others. 

\subsection*{Notation convention} In the following, we are primarily working on a fixed Riemannian surface $(M,g)$; as a result, the mention of the metric $g$ is often suppressed in the notation. For example, integration over $M$ is always with respect to the volume measure $dv_g$ unless stated otherwise, the functional spaces $L^2(M)$ and $W^{1,2}(M)$ refer to $L^2(M,g)$ and $W^{1,2}(M,g)$ respectively, etc.

\subsection*{Acknowledgements} The authors would like to thank Iosif Polterovich and Jean Lagac\'e for remarks on the preliminary version of the manuscript. 

This project originated during the CRG workshop on Geometric Analysis held at the University of British Columbia in May 2019. The hospitality of the University of British Columbia is gratefully acknowledged.

\section{Preliminaries}
\label{prelim:sec}
\subsection{Harmonic maps to $\mathbb{S}^n$}
\label{intro_harm:sec}
Recall that a map $\Phi\colon (M,g) \to (N,h)$ between Riemannian manifolds is said to be harmonic if it is a critical point of the energy functional
$$
E_g(\Phi) = \frac{1}{2}\int_M|d\Phi|^2_{g,h}\,dv_g.
$$
When the domain is a surface $(M,g)$, the energy $E_g(\Phi)$ is conformally invariant with respect to the metric $g$, and it follows that a map $\Phi:M\to N$ which is harmonic for $g$ is also harmonic for any conformal metric $\tilde{g}\in [g]$. In the following we fix a conformal class $c=[g]$ on a surface $M$ and often suppress the metric in the notation of any conformally invariant object.

When the target $(N,h)$ is the unit sphere $\mathbb{S}^n\subset\mathbb{R}^{n+1}$, a standard computation shows that a map $\Phi\colon (M,g) \to \mathbb{S}^n$ is harmonic if and only if it satisfies the equation
\begin{equation}
\label{harmonic:eq}
\Delta_g\Phi = |d\Phi|_g^2\Phi.
\end{equation}
In particular, letting $g_\Phi = \frac{1}{2}|d\Phi|_g^2g$, then using the conformal covariance of $\Delta_g$, the equation~\eqref{harmonic:eq} becomes
\begin{equation}
\label{harmonic:eq2}
\Delta_{g_\Phi}\Phi = 2\Phi;
\end{equation}
i.e. the components of $\Phi$ are eigenfunctions of $\Delta_{g_\Phi}$ with eigenvalue $2$. 
\begin{remark}
\label{conical:remark}
Note that $|d\Phi|_g^2$ can vanish at isolated points. At such points it is said that $g_\Phi$ has an isolated conical singularity. These are fairly mild singularities, and the eigenvalues can be defined in the same way using the Rayleigh quotient, see Remark~\ref{conical:remark2} below or~\cite{CKM}.
\end{remark}


\begin{definition}
For a harmonic map $\Phi\colon M\to\mathbb{S}^n$, the {\em spectral index}  $\ind_S(\Phi)$ is defined to be the minimal $k\in \mathbb{N}$ such that $\lambda_k(M,g_\Phi)=2$. Equivalently, $\ind_S(\Phi)$ is the index of the quadratic form
$$
Q_S(u) = \int|du|_g^2 - |d\Phi|_g^2u^2\,dv_g
$$
over $u\in W^{1,2}(M,\mathbb{R})$ for some (any) metric $g\in c$.
\end{definition}
\begin{definition}
Likewise, the {\em spectral nullity}  $\nul_S(\Phi)$ is the multiplicity of eigenvalue $2$ for the metric $g_\Phi$. Alternatively, $\nul_S(\Phi)$ is the nullity of the quadratic form $Q_S$
for some (any) metric $g\in c$.
\end{definition}

Note that with this definition one always has 
$$\bar\lambda_{\ind_S(\Phi)}(M,g_\Phi) = 2E(\Phi).$$
With this notation in place, we can now give a more precise statement of Theorem~\ref{extremalThm:intro}

\begin{theorem}[Nadirashvili~\cite{NadirashviliTorus}, El Soufi-Ilias~\cite{ESI}, see also~\cite{FS2}]
Suppose that $h\in c$ is such that $\bar\lambda_k(M,h) = \Lambda_k(M,c)$. Then there exists a harmonic map $\Phi\colon M\to\mathbb{S}^n$ such that $\ind_S(\Phi) = k$ and $h=\alpha g_\Phi$, where $\alpha>0$ is a constant.
\end{theorem}

The existence of a metric $h\in c$ realizing the supremum $\bar\lambda_k(M,h) = \Lambda_k(M,c)$ is not always guaranteed. For example, there is no maximal metric for the second eigenvalue on the sphere, see~\cite{PetridesS2, NadirashviliS2, KNPP}. One obvious obstruction is provided in~\cite{CES}, where it is shown that given $k$ one can glue a sphere to any metric $g$ on $M$, without changing the conformal class, to obtain a metric $g'\in[g]$ satisfying $\bar\lambda_k(M,g') = \bar\lambda_{k-1}(M,g)+8\pi$. In particular, setting $\Lambda_0(M,c)=0$, one has 
\begin{equation}
\label{Lambdak:condition}
\Lambda_k(M,c)\geqslant \Lambda_{k-1}(M,c) + 8\pi,
\end{equation}
where the case of equality suggests the appearance of spherical `bubbles' along a maximizing sequence. Fortunately, it turns out that equality in \ref{Lambdak:condition} is the only obstruction to the existence of a maximizing metric.
\begin{theorem}
[Petrides~\cite{Petrides, Petridesk}, K.-Nadirashvili-Penskoi-Polterovich~\cite{KNPP, KNPP2}]
\label{existence:thm}\hspace{4mm}
If the inequality~\eqref{Lambdak:condition} is strict, then there exists $h\in c$ such that $\bar\lambda_k(M,h) = \Lambda_k(M,c)$. In particular, $h=\alpha g_\Phi$ for some harmonic map $\Phi\colon (M,g)\to\mathbb{S}^n$ of spectral index $k$.
\end{theorem}

In fact, for $k=1$ one can say more.
\begin{theorem}[Petrides~\cite{Petrides}]
\label{rigidity:thm}
Suppose that $M$ is not a $2$-dimensional sphere $\mathbb{S}^2$. Then for any conformal class $c$ on $M$ one has
$$
\Lambda_1(M,c)>8\pi,
$$
i.e. the inequality~\eqref{Lambdak:condition} is strict.
\end{theorem}

One of the byproducts of the min-max characterisation is an alternative proof of Theorem~\ref{existence:thm} for $k=1,2$. 

\subsection{Energy index} Given a harmonic map $\Phi$, the energy index $\ind_E(\Phi)$ refers to the Morse index of $\Phi$ as a critical point of energy functional. More concretely, we have the following definition.

\begin{definition}
Let $\Phi\colon M\to\mathbb{S}^n\subset \mathbb{R}^{n+1}$ be a harmonic map. Then for any metric $g\in c$, the {\em energy index} $\ind_E(\Phi)$ is given by the index of the quadratic form
$$
Q_E(V) =  \int|dV|_g^2 - |d\Phi|_g^2|V|^2\,dv_g
$$
over sections of the pullback bundle
$$\Gamma(\Phi^*(T\mathbb{S}^n))\cong\{V:M\to \mathbb{R}^{n+1}\mid V(x)\perp \Phi(x)\text{ for each }x\in M\}.$$
\end{definition}

The first author has shown in~\cite{KarRP2} that the energy index and spectral index are closely related. For the purposes of the present article, we only need the following two results from~\cite{KarRP2}.

\begin{proposition}
\label{indEindS:prop}
Let $\Phi\colon M\to\mathbb{S}^n$ be a harmonic map and let $i_{n,m}\colon \mathbb{S}^n\to\mathbb{S}^m$ be a totally geodesic embedding, $m\geqslant n$. Then one has
$$
\ind_E(i_{n,m}\circ\Phi) = \ind_E(\Phi) + (m-n)\ind_S(\Phi).
$$
\end{proposition}

\begin{proposition}
\label{indES2:prop}
Let $\Phi\colon\mathbb{S}^2\to\mathbb{S}^n$ be a nonconstant harmonic map. Suppose that $\Phi$ is not a totally geodesic embedding, i.e. it is not an embedding into an equatorial $\mathbb{S}^2\subset\mathbb{S}^n$. Then one has
$$
\ind_E(\Phi)>n+1
$$
if $n>4$ and
$$
\ind_E(\Phi)>2(n+1)
$$
for $n>8$.
\end{proposition}
\begin{proof}
The proof is an easy application of results in~\cite{KarRP2}. Namely, in~\cite{KarRP2} it is shown that
$$
\ind_E(\Phi)\geqslant (n-2)(2d - [\sqrt{8d+1}]_{\mathrm{odd}}+2),
$$
where $d=\frac{E(\Phi)}{4\pi}$ is the degree of $\Phi$ and $[x]_{\mathrm{odd}}$ denotes the smallest odd number not exceeding $x$. By a result of Barbosa~\cite{Barbosa}, $d\in\mathbb{N}$. We claim that for $d>1$ one has $(2d - [\sqrt{8d+1}]_{\mathrm{odd}}+2)\geqslant 3$. Indeed, If $d\geqslant 4$ then $8d+1<16(d-3/2)$ so that
$$
(2d - [\sqrt{8d+1}]_{\mathrm{odd}}-1)\geqslant (2d - 4\sqrt{d-3/2}-1) =2\left(\sqrt{d-3/2} - 1\right)^2\geqslant 0.
$$
If $d=2,3$, then an explicit computation shows that $2d - [\sqrt{8d+1}]_{\mathrm{odd}}+2$ equals $3$.

At the same time, $d=1$ iff $\Phi$ is a totally geodesic embedding $\mathbb{S}^2\to\mathbb{S}^n$.
Therefore, if $\Phi$ is not totally geodesic, then
$$
\ind_E(\Phi)\geqslant 3(n-2).
$$
The proof is completed by noting that $3(n-2)>n+1$ for $n>4$ and $3(n-2)>2(n+1)$ for $n>8$.
\end{proof}

Finally, we recall the following result of El Soufi.
\begin{proposition}[El Soufi~\cite{ESindex}]
\label{indEM:prop}
Let $\Phi\colon M\to\mathbb{S}^n$ be a non-constant harmonic map. Then 
$$
\ind_E(\Phi)\geqslant n-2.
$$
\end{proposition}
\subsection{Eigenvalues of Radon measures}\label{meas.eigenvals} In the paper~\cite{Kokarev}, Kokarev defines a natural analog $\lambda_k(M,c,\mu)$ of Laplacian eigenvalues associated with a surface $M$, a conformal class $c$, and a Radon measure $\mu$ on $M$. If $g\in c$, then the ``eigenvalue" $\lambda_k(M,c,\mu)$ is defined via the Rayleigh quotient, by
\begin{equation}
\label{MeasureRayleigh:quotient}
\lambda_k(M,c,\mu) := \inf_{F_{k+1}}\sup_{u\in F_{k+1}\setminus\{0\}}\frac{\displaystyle\int_M|\nabla u|^2_g\,dv_g}{\displaystyle\int_M u^2\,d\mu},
\end{equation}
where the infimum is taken over all $(k+1)$-dimensional subspaces $F_{k+1}\subset C^\infty(M)$ that remain $(k+1)$-dimensional in $L^2(M,\mu)$. 

Note that for the standard volume measure $\mu=dv_g$, the associated eigenvalues $\lambda_k(M,[g],dv_g)=\lambda_k(M,g)$ coincide with the classical Laplacian eigenvalues. However, there are several other classes of measures $\mu$ whose associated eigenvalues $\lambda_k(M,c,\mu)$ are of geometric interest; in Section \ref{steklov.sec}, for example, we will be particularly interested in the case when $\mu$ is the length measure $\mu=\mathcal{H}^1|_\Gamma$ associated to a closed curve $\Gamma$ in $M$. This definition also makes sense if $M$ is replaced by a surface $\Omega$ with non-empty boundary. In this case $\lambda_k(\Omega,[g],dv_g)$ are the {\em Neumann} eigenvalues of $\Omega$ and $\lambda_k(\Omega,[g],\mathcal H^1|_{\partial\Omega})$ are the Steklov eigenvalues.

\begin{remark}
\label{conical:remark2} Consider a degenerate conformal metric $h=fg$, where $f\in C^{\infty}(M)$ satisfies $f>0$ outside of possibly finitely many isolated zeroes, as in Remark~\ref{conical:remark}. Then we can take the quantities $\lambda_k(M,[g],f\,dv_g)$ given by \eqref{MeasureRayleigh:quotient} as the definition of the Laplacian eigenvalues for the degenerate metric $h$; see~\cite[Example 1.1]{Kokarev}.
\end{remark}

In~\cite{Kokarev} Kokarev studied extremal properties of eigenvalues $\lambda_k(M,c,\mu)$ over the space of Radon probability measures, and was able to prove partial regularity results for maximizers under a certain mild regularity assumption on the measures $\mu$. Below we follow the exposition in~\cite[Section 3]{GKL}. Let $\mathcal L$ be a completion of $C^\infty(M)$ with respect to the norm 
$$
||u||^2_{\mathcal L} = \int u^2\,d\mu + \int |du|_g^2\,dv_g = ||u||^2_{L^2(M,\mu)} + ||\nabla u||^2_{L^2(M,g)}.
$$ 

\begin{definition}
We call a Radon measure $\mu$ {\em admissible} if the identity map on $C^\infty(M)$ extends to a compact map $T\colon W^{1,2}(M,g) \to L^2(M,\mu)$. 
\end{definition}

\begin{proposition}
\label{ef:prop}
Let $\mu$ be an admissible measure. Then the identity map on $C^\infty(M)$ extends to a bounded isomorphism between $\mathcal L$ and $W^{1,2}(M,g)$. Furthermore, one has 
$$
0=\lambda_0(M,[g],\mu)<\lambda_1(M,[g],\mu)\leqslant \lambda_2(M,[g],\mu)\leqslant\ldots\nearrow\infty;
$$
i.e. the first eigenvalue is positive, the multiplicity of each eigenvalue is finite, and the eigenvalues tend to $+\infty$. Moreover, each eigenvalue $\lambda_i(M,[g],\mu)$ has an associated eigenfunction $\phi_i\in \mathcal L$ satisfying
\begin{equation}
\label{eigenfunctions_measures:def}
\int\langle \nabla\phi_i,\nabla u\rangle\,dv_g = \lambda_i(M,[g],\mu)\int \phi_i u\,d\mu
\end{equation}
for all $u\in \mathcal L$.
\end{proposition}
\begin{proof}
See~\cite[Section 3]{GKL} for the proof.
\end{proof}

\begin{proposition}
\label{ef_Lp:prop}
Let $\mu$ be an absolutely continuous Radon measure $d\mu = f\,dv_g$, where $f\in L^p(M,g)$, $f\geqslant 0$, $p>1$. Then $\mu$ is admissible.
\end{proposition}
\begin{proof}
One can find a proof in~\cite[Example 2.1]{Kokarev}. We provide a simpler proof for completeness. For any $u\in C^\infty(M)$, an easy application of H\"older's inequality gives
$$
||u||^2_{L^2(M,\mu)} = \int u^2 f\,dv_g\leqslant ||u||^2_{L^{2q}(M,g)} ||f||_{L^p(M,g)},
$$
where $q$ is the H\"older conjugate of $p$. Hence, the identity map on $C^\infty(M)$ extends to a bounded map $L^{2q}(M,g)\to L^2(M,\mu)$, and since the embedding $W^{1,2}(M,g)\to L^{2q}(M,g)$ is compact by Rellich's theorem, the admissibility of $\mu$ follows.
\end{proof}




\subsection{Conformal volume}\label{vc.def} Our min-max construction is inspired in large part by the notion of conformal volume introduced by P. Li and S.-T. Yau in~\cite{LiYau}, which we briefly review. Recall that the group of conformal automorphisms of $\mathbb{S}^n$ modulo $O(n+1)$ is homeomorphic to the open ball $B^{n+1}$. To each $a\in B^{n+1}$ there corresponds 
a conformal automorphism $G_a$ given by
$$
G_a(x)=\frac{(1-|a|^2)}{|x+a|^2}(x+a)+a,
$$
Let $\phi\colon M\to\mathbb{S}^n$ be a conformal immersion, then one successively defines
$$
V_c(n,\phi) = \sup_aE(G_a\circ \phi) = \sup_a \area(G_a\circ \phi)
$$   
$$
V_c(n,M,[g])  = \inf_\phi V_c(n,\phi)
$$
$$
V_c(M,[g]) = \lim_{n\to\infty} V_c(n,M,[g]) = \inf_n V_c(n,M,[g]).
$$
One of the applications of conformal volume obtained in~\cite{LiYau} is the following.
\begin{proposition}[Li, Yau~\cite{LiYau}]
\label{Vc:prop}
One has 
$$
\Lambda_1(M,[g])\leqslant 2V_c(M,[g])<+\infty.
$$
\end{proposition}
\begin{remark}
Li and Yau obtained an upper bound for the conformal volume for orientable surfaces, see~\cite{KarNonOrientable, KokarevVc} for the non-orientable case. Notably, the quantity $V_c(n,M,[g])$ is also bounded above by the Willmore energy of conformal immersions $(M,g)\to \mathbb{S}^n$, making the conformal volumes an important tool in the study of the Willmore functional, in addition to their role as a source of eigenvalue estimates.
\end{remark}
Our definition of the min-max energy $\mathcal E_n(M,[g])$ below may be regarded as a maximal possible relaxation of the notion of $V_c(n,M,[g])$ so that the proof of Proposition~\ref{Vc:prop} still holds.

\section{The min-max construction for the first eigenvalue}\label{first.min-max}


When producing harmonic maps $M\to N$ via variational methods, instead of working directly with the Dirichlet energy on the space $W^{1,2}(M,N)$, it is often simpler to first produce critical points for a sequence of perturbed functionals on different function spaces (with better regularity and compactness properties), which then limit to a harmonic map up to bubbling phenomena. The first example of this approach comes from the work of Sacks and Uhlenbeck \cite{SU}, who produced harmonic maps from closed surfaces into higher-dimensional targets by applying variational methods to a family of perturbed functionals (essentially the $L^p$ norm of the gradient) on the spaces $W^{1,p}(M,N)\subset C^0(M,N)$ for $p>2$. 

Since the families of maps to which we wish to apply min-max methods are not continuous in the $C^0$ or $W^{1,2}$ topologies, the classical Sacks-Uhlenbeck perturbation--which penalizes bubbling behavior with infinite energy--is not quite suitable for our needs. Instead, we employ a \emph{relaxation} of the harmonic map problem via functionals of ``Ginzburg-Landau" type, building on the analysis of \cite{CS,CL,LW}.

More precisely, for small positive $\epsilon>0$, we consider the functionals
$$E_{\epsilon}: W^{1,2}(M,\mathbb{R}^{n+1})\to \mathbb{R}$$
defined on vector-valued maps $u: M\to \mathbb{R}^{n+1}$ by 
\begin{equation*}
E_{\epsilon}(u):=\int_M\frac{1}{2}|du|^2+\frac{1}{4\epsilon^2}(1-|u|^2)^2.
\end{equation*}
Note that for a map $u: M\to \mathbb{S}^n\subset \mathbb{R}^{n+1}$ taking values in the unit sphere, the functional $E_{\epsilon}$ recovers the Dirichlet energy $E_{\epsilon}(u)=E(u)=\frac{1}{2}\int_M |du|^2$, while for general maps to $\mathbb{R}^{n+1}$, the nonlinear potential term $\frac{(1-|u|^2)^2}{4\epsilon^2}$ penalizes deviation from $\mathbb{S}^n$, with increasing severity as $\epsilon\to 0$.

In the following proposition, we note that the functionals $E_{\epsilon}$ satisfy all requisite properties for the construction of critical points via classical min-max methods.

\begin{proposition}\label{ps.fred.prop} The functionals $E_{\epsilon}$ defined above are $C^2$ functionals on $W^{1,2}(M,\mathbb{R}^{n+1})$, with first and second derivatives given by
$$\langle E_{\epsilon}'(u),v\rangle=\int_M\langle du,dv\rangle-\epsilon^{-2}(1-|u|^2)\langle u,v\rangle$$
and
$$\langle E_{\epsilon}''(u),v\rangle=\int_M\Delta_gv+2\epsilon^{-2}\langle u,v\rangle u-\epsilon^{-2}(1-|u|^2)v.$$
Moreover, the second derivative $E_{\epsilon}''(u)$ defines a Fredholm operator at critical points $u$ of $E_{\epsilon}$, and the functionals $E_{\epsilon}$ satisfy the standard Palais-Smale compactness condition: for any sequence $u_j\in W^{1,2}(M,\mathbb{R}^{n+1})$ such that
$$\sup_j E_{\epsilon}(u_j)<\infty\text{ and }\lim_{j\to\infty}\|E_{\epsilon}'(u_j)\|_{(W^{1,2})^*}=0,$$
there exists a subsequence $u_{j_k}$ that converges strongly in $W^{1,2}(M,\mathbb{R}^{n+1})$.
\end{proposition}

The proof of these properties is a standard exercise; for details, the reader may consult, e.g., Section 4 of \cite{Gu} and Section 7 of \cite{St.pharm}, and references therein.

\subsection{Definition and estimates for the first min-max energies} \hspace{40mm}
\label{first.min-max:def}

Given a closed Riemannian surface $(M,g)$ and $n\geq 2$, we will denote by $\Gamma_n(M)$ the collection of all families
$$F\in C^0(\overline{B}^{n+1},W^{1,2}(M,\mathbb{R}^{n+1}))\text{ such that }F_a\equiv a\text{ for }a\in \mathbb{S}^n.$$
Then, for $\epsilon>0$, we define the min-max energy
\begin{equation}
\label{GL:example}
\mathcal{E}_{n,\epsilon}(M,g):=\inf_{F\in \Gamma_n(M)}\max_{a\in \overline{B}^{n+1}}E_{\epsilon}(F_a).
\end{equation}
Noting that the energies $\mathcal{E}_{n,\epsilon}$ are decreasing functions of $\epsilon$, we also define the limit
\begin{equation}
\mathcal{E}_n(M,g):=\sup_{\epsilon>0}\mathcal{E}_{n,\epsilon}(M)=\lim_{\epsilon\to 0}\mathcal{E}_{n,\epsilon}(M).
\end{equation}

Observe that, while the perturbed min-max energies $\mathcal{E}_{n,\epsilon}$ are not conformally invariant, the limiting energy $\mathcal{E}_n(M,g)$ is independent of the conformal representative $g\in [g]$. This follows from the simple observation that, for any fixed metrics $g,\tilde{g}\in [g]$, we have $\,dv_g\leq C^2dv_{\tilde{g}}$ for some positive constant $C=C(g,\tilde{g})$, and since the Dirichlet energy is conformally invariant, it follows from the definition of the functionals $E_{\epsilon}$ that
$$E_{\epsilon}(u,g)\leq E_{\epsilon/C}(u,\tilde{g})\text{ for all }u\in W^{1,2}(M,\mathbb{R}^{n+1}).$$
In particular, we have $\mathcal{E}_{n,\epsilon}(M,g)\leq \mathcal{E}_{n,\epsilon/C}(M,\tilde{g})$, and taking the limit as $\epsilon\to 0$ gives 
$$\mathcal{E}_n(M,g)\leq \mathcal{E}_n(M,\tilde{g})$$
for arbitrary conformal metrics $g,\tilde{g}\in [g]$. Henceforth we write 
$$\mathcal{E}_n(M,[g]):=\mathcal{E}_n(M,g).$$

\begin{remark} In the scalar-valued ($n=0$) and complex-valued ($n=1$) cases, the min-max energies $\mathcal{E}_{n,\epsilon}$ and the associated critical points have previously been studied in \cite{Gu} and \cite{Stern}. Though formally identical, these constructions are qualitatively quite different from the case $n\geq 2$ considered here, with energy blowing up as $\epsilon\to 0$, and associated critical points exhibiting energy concentration along (generalized) minimal submanifolds in $M$ of codimension one and two, respectively.
\end{remark}

In the following proposition, we show that the limiting min-max energies $\mathcal{E}_n(M,[g])$ are finite (for $n\geq 2$), and in particular are bounded above by the Li--Yau conformal volume $V_c(n,M,[g])$.

\begin{proposition}\label{vc.above} For each $n\geq 2$, we have
$$\mathcal{E}_n(M,[g])\leq V_c(n,M,[g])<\infty.$$
\end{proposition}
\begin{proof} As in Section \ref{vc.def}, let $\phi: (M^2,g)\to \mathbb{S}^n$ be a branched conformal immersion, and for $a\in \overline{B}^{n+1}$, let $G_a: \mathbb{S}^n\to \mathbb{S}^n$ be the conformal map
$$G_a(x)=\frac{(1-|a|^2)}{|x+a|^2}(x+a)+a.$$
Denote by $F_a: M\to \mathbb{S}^n$ by composition
\begin{equation}\label{canon.def}
F_a:=G_a\circ \phi.
\end{equation}
The maximum Dirichlet energy of $F_a$ over $a\in \overline{B}^{n+1}$ is then given by
\begin{equation}
V_c(n,\phi)=\sup_{|a|\leq 1}E(F_a)=\sup_{|a|\leq 1}\area(F_a(M)).
\end{equation}
The family $a\mapsto F_a$ is only continuous with respect to the weak topology on $W^{1,2}(M,\mathbb{S}^n)$ as $|a|\to 1$, but we can mollify it to produce a continuous family in the strong topology on $W^{1,2}(M,\mathbb{R}^{n+1})$.

To this end, denote by $K_t(x,y)$ the heat kernel on $M$, and for $t>0$, consider the mollifying map
$$\Phi^t: L^1(M,\mathbb{R}^{n+1})\to C^{\infty}(M,\mathbb{R}^{n+1})$$
given by
$$(\Phi^tF)(x):=\int_M F(y)K_t(x,y)dy.$$
Note that $\Phi^t$ fixes the constant maps, and by the smoothness of $K_t$, $\Phi^t$ is continuous as a map from $L^1$ to $C^1$, since for any maps $F_1,F_2\in L^1(M,\mathbb{R}^{n+1})$,
\begin{eqnarray*}
|d(\Phi^tF_1)(x)-d(\Phi^tF_2)(x)|&=&|\int_M (F_1(y)-F_2(y)) d_xK_t(x,y) dy|\\
&\leq &|K_t|_{C^1}\|F_1-F_2\|_{L^1}.
\end{eqnarray*}

In particular, since the family $B^{n+1}\ni a\mapsto F_a$ given by \eqref{canon.def} is continuous as a map into $L^1(M,\mathbb{R}^{n+1})$, it follows that the mollified family
$$F^t_a(x):=(\Phi^tF_a)(x)=\int_M F_a(y)K_t(x,y)dy$$
defines a continuous family in $W^{1,2}(M,\mathbb{R}^{n+1})$, which belongs moreover to $\Gamma_n(M)$, since
$$(\Phi^tF_a)\equiv F_a\equiv a\text{ for }a\in \mathbb{S}^n.$$
Moreover, since $t\mapsto F^t_a$ solves the heat equation $\frac{\partial F^t_a}{\partial t}=\Delta F^t_a$ with initial data $F_a^0=F_a$, it follows immediately that
\begin{equation}\label{dir.bd}
\int_M \frac{1}{2}|dF_a^t|^2\leq \int_M \frac{1}{2}|dF_a|^2\leq V_c(n,\phi)
\end{equation}
for all $t\geq 0$. 

Next, we claim that
\begin{equation}\label{pot.van}
\delta(t):=\max_{a\in \overline{B}^{n+1}}\int_M (1-|F_a^t|^2)^2\to 0\text{ as }t\to 0.
\end{equation}
Indeed, if this were false, then we could find a sequence $t_j\to 0$ and $a_j\in \overline{B}^{n+1}$ such that
$$\lim_{j\to\infty}\int_M (1-|F_{a_j}^{t_j}|^2)^2>0.$$
But, passing to a subsequence, we also have $a_j\to a$ for some $a\in \overline{B}^{n+1}$, and it follows readily from the definition of the families $F_a^t$ that $F_{a_j}^{t_j}\to F_a$ in $L^p$ as $j\to\infty$ for any $p\in [1,\infty)$. Since $|F_a|\equiv 1$ pointwise, it then follows that 
$$\lim_{j\to\infty}\int_M(1-|F_{a_j}^{t_j}|^2)^2=0$$
after all, confirming the claim \eqref{pot.van}.

For any fixed $\epsilon>0$, it now follows from the observations above that
\begin{equation}
\mathcal{E}_{n,\epsilon}(M,g)\leq \lim_{t\to 0}\max_{|a|\leq 1}E_{\epsilon}(F_a^t)\leq V_c(n,\phi),
\end{equation}
as desired. Taking the infimum over all branched conformal immersions $\phi:M\to \mathbb{S}^n$, we obtain the upper bound
$$\mathcal{E}_{n,\epsilon}(M,g)\leq V_c(n,M,[g]).$$
Finally, taking the supremum over all $\epsilon>0$ gives
$$\mathcal{E}_n(M,[g])\leq V_c(n,M,[g]),$$
as desired.
\end{proof}

Next, we come to the key lower bound for the min-max energies $\mathcal{E}_n(M,c)$, showing that they dominate the normalized first Laplacian eigenvalue of any conformal metric $g\in c$. Later, in Theorem \ref{mu.eigen.bd}, we obtain a strengthened version of the following inequality, showing that $2\mathcal{E}_n$ provides an upper bound for the first eigenvalue for a more general class of probability measures.

\begin{proposition}\label{lambda.below} If $Area(M,g)=1$, then 
\begin{equation}
2\mathcal{E}_{n,\epsilon}(M,g)\geq (1-2\epsilon\mathcal{E}_{n,\epsilon}^{1/2})\lambda_1(M,g).
\end{equation}
In particular, 
$$2\mathcal{E}_n(M,[g])\geq \Lambda_1(M,[g]).$$
\end{proposition}
\begin{proof} The proof follows from a standard trick, essentially equivalent to that used by Li--Yau in \cite{LiYau}. Given $F\in \Gamma_n(M)$, consider the continuous map $f:\overline{B}^{n+1}\to \mathbb{R}^{n+1}$ given by taking the average
$$f(a)=\int_MF_a.$$
By definition of $\Gamma_n(M)$, we then see that $f|_{\partial B^{n+1}}=\mathrm{Id}\colon \mathbb{S}^n\to \mathbb{S}^n$ is homotopically nontrivial on the boundary sphere $\partial B^{n+1}$, and it follows that
\begin{equation}\label{0.avg}
\int_MF_a=0\in \mathbb{R}^{n+1}
\end{equation}
for some $a\in B^{n+1}$. By the variational characterization of the first eigenvalue $\lambda_1(M)$, we then have
\begin{equation}\label{poincare}
\lambda_1(M)\int_M|F_a|^2\leq \int_M|dF_a|^2\leq 2\max_{a\in \overline{B}^{n+1}}E_{\epsilon}(F_a)
\end{equation}
for $a\in B^{n+1}$ satisfying \eqref{0.avg} holds. Moreover, it follows from the definition of $E_{\epsilon}$ that
\begin{eqnarray*}
\int_M |F_a|^2&\geq & 1-\int_M|1-|F_a|^2|\\
&\geq & 1-2\epsilon E_{\epsilon}(F_a)^{1/2}\\
&\geq & 1-2\epsilon \max_{a\in \overline{B}^{n+1}}E_{\epsilon}(F_a)^{1/2}.
\end{eqnarray*}
Putting this together with \eqref{poincare}, we arrive at the desired estimate by choosing families $F\in \Gamma_n(M)$ such that $\max_a E_{\epsilon}(F_a)$ is arbitrarily close to $\mathcal{E}_{n,\epsilon}$. The estimate 
$$2\mathcal{E}_n(M,[g])\geq \Lambda_1(M,[g])$$
then follows by taking $\epsilon\to 0$, and invoking the conformal invariance of $\mathcal{E}_n(M,[g])$.
\end{proof}

\subsection{Existence and properties of the min-max harmonic maps} \hspace{10mm}
\label{first.min-max:properties}

Since the functionals $E_{\epsilon}$ satisfy the technical requirements laid out in Proposition \ref{ps.fred.prop}, and the collection $\Gamma_n(M)$ of $(n+1)$-parameter families in $W^{1,2}(M,\mathbb{R}^{n+1})$ is evidently preserved by the gradient flow of $E_{\epsilon}$, we can appeal to standard results in critical point theory (see, e.g., Chapter 10 of \cite{Ghou}, in particular Corollary 10.16) to arrive at the following existence result for each $\epsilon>0$.

\begin{proposition}\label{min.max} There exists a critical point $\Psi_{\epsilon}:(M,g)\to \mathbb{R}^{n+1}$ for $E_{\epsilon}$ of energy
\begin{equation}
E_{\epsilon}(\Psi_{\epsilon})=\mathcal{E}_{n,\epsilon},
\end{equation}
satisfying the Morse index bound
\begin{equation}
\ind_{E_{\epsilon}}(\Psi_{\epsilon})\leq n+1.
\end{equation}
\end{proposition}

Our goal now is to deduce the existence of a harmonic map $\Psi:M\to \mathbb{S}^n$, of energy-index $\ind_E(\Psi)\leq n+1$, given as the strong $W^{1,2}$-limit of the critical points constructed in Proposition \ref{min.max}. To this end, we introduce the following technical lemma, combining the bubbling analysis of \cite{LW} with a lower semi-continuity result for the Morse index, modeled on analogous results (cf. \cite{MooreReam}) for the Sacks-Uhlenbeck perturbation.

\begin{lemma}\label{bubble.business} Let $\{\Psi_{\epsilon}\}$ be a family of critical points $\Psi_{\epsilon}:M\to \mathbb{R}^{n+1}$ for the energy $E_{\epsilon}$, satisfying 
\begin{equation}
\label{ener.bd}
\Lambda:=\lim_{\epsilon \to 0}E_{\epsilon}(\Psi_{\epsilon})<\infty
\end{equation}
and the Morse index bound
\begin{equation}
\label{index.bd}
\ind_{E_{\epsilon}}(\Psi_{\epsilon})\leq m.
\end{equation}
Then for a subsequence $\epsilon_j\to 0$, there exists a collection of points $\{a_1,\ldots,a_{\ell}\}\subset M$, a harmonic map $\Psi: M\to \mathbb{S}^n$, and harmonic maps $\phi_1,\ldots,\phi_k:\mathbb{S}^2\to \mathbb{S}^n$ such that
$$\Psi_{\epsilon_j}\to \Psi\text{ in }C^2_{loc}(M\setminus \{a_1,\ldots,a_{\ell}\})\text{ and weakly in }W^{1,2}(M,\mathbb{R}^{n+1}),$$
for which we have the energy identity
\begin{equation}\label{ener.id}
\Lambda=E(\Psi)+\sum_{j=1}^kE(\phi_j)
\end{equation}
and the energy-index bound
\begin{equation}\label{index.semicont}
\ind_E(\Psi)+\sum_{j=1}^k\ind_E(\phi_j)\leq m.
\end{equation}
\end{lemma}

\begin{proof} The existence of the limiting harmonic map $\Psi:M\to \mathbb{S}^n$ and bubbles $\phi_j:\mathbb{S}^2\to \mathbb{S}^n$ satisfying the energy identity \eqref{ener.id} is contained already in the work of Lin and Wang (see \cite{LW}, Theorem A), so the only point that requires comment is the statement \eqref{index.semicont} concerning lower semi-continuity of the index along the bubble tree.

Most of the details of the proof of \eqref{index.semicont} can be borrowed directly from the proof of the identical statement for the Sacks-Uhlenbeck perturbation of the harmonic mapping problem (see, e.g., \cite{MooreReam}, Theorem 6.2). In the interest of completeness, we review the main technical ingredient (restricting variational vector fields to the complement of bubbling regions) in our setting, proving the following claim.

\begin{claim} Given a family of maps $\Psi_{\epsilon}:M^2\to \mathbb{R}^{n+1}$, a collection of points $\{a_1,\ldots,a_k\}\subset M$, and a vanishing sequence of radii $r_{\epsilon}\to 0$ such that 
\begin{equation}\label{gl.eq}
\epsilon^2\Delta \Psi_{\epsilon}=(1-|\Psi_{\epsilon}|^2)\Psi_{\epsilon}\text{\hspace{2mm} on }M\setminus \bigcup_{j=1}^kD_{r_{\epsilon}}(a_j),
\end{equation}
suppose that there exists a harmonic map $\Psi:M\to \mathbb{S}^n$ such that $\Psi_{\epsilon}\to \Psi$ in $C^2_{loc}(M\setminus \{a_1,\ldots,a_k\})$ as $\epsilon\to 0$. Then the $E_{\epsilon}$-index $ind_{E_{\epsilon}}(\Psi_{\epsilon})$ of $\Psi_{\epsilon}$ with respect to variations supported in $M\setminus \bigcup_{j=1}^kD_{r_{\epsilon}}(a_j)$ is at least as large as the energy-index $ind_E(\Psi)$ of $\Psi$ on $M$, for $\epsilon>0$ sufficiently small.
\end{claim}

Once this claim is in place, we can argue exactly as in \cite{MooreReam}, applying the claim at each node in the bubble tree for the family $\{\Psi_{\epsilon}\}$, to complete the proof of \eqref{index.semicont}.

To prove the claim, recall that the second variation $Q_E(\Psi)$ of energy about the harmonic map $\Psi$ is given by
$$Q_E(\Psi)(V,V):=\int_M |dV|^2-|d\Psi|^2|V|^2$$
for maps $V: M\to \mathbb{R}^{n+1}$ with $\langle \Psi,V\rangle \equiv 0$ on $M$. Likewise, as we've seen in Proposition \ref{ps.fred.prop}, for an $\mathbb{R}^{n+1}$-valued map $V$ supported in the domain of $\Psi_{\epsilon}$, the second variation $Q_{E_{\epsilon}}(\Psi_{\epsilon})$ of $E_{\epsilon}$ at $\Psi_{\epsilon}$ is given by
$$Q_{E_{\epsilon}}(\Psi_{\epsilon})(V,V)=\int_M(|dV|^2+2\epsilon^{-2}\langle \Psi_{\epsilon},V\rangle^2-\epsilon^{-2}(1-|\Psi_{\epsilon}|^2)|V|^2.$$
Let $p=\ind_E(\Psi)$; then there exists a $p$-dimensional subspace $\mathcal{V}\subset \Gamma(\Psi^*(T\mathbb{S}^n))$ and $\beta>0$ such that
\begin{equation}
Q_E(\Psi)(V,V)<-\beta\|V\|_{L^2}^2\text{ for every }0\neq V\in \mathcal{V}.
\end{equation}
As in \cite{MooreReam}, we employ logarithmic cutoff functions to perturb this subpace $\mathcal{V}$ to a new subspace $\widetilde{\mathcal{V}}$ of variations vanishing on the disks $D_r(a_1)\cup\cdots D_r(a_k)$ for $r>0$ sufficiently small, such that
\begin{equation}\label{pert.ind.bd}
Q_E(\Psi)(V,V)<-\frac{\beta}{2}\|V\|_{L^2}^2\text{ for every }0\neq V\in \widetilde{\mathcal{V}}.
\end{equation}

Specifically, for $\delta>0$, define $\phi_{\delta}:\mathbb{R}\to\mathbb{R}$ by
$$\phi_{\delta}(t)=2-\frac{\log(t)}{\log(\delta)}\text{ for }t\in [\delta^2,\delta],$$
while $\phi_{\delta}(t)=0$ for $t\leq \delta^2$ and $\phi_{\delta}(t)=1$ for $t\geq \delta$. Then define the cutoff functions $\psi_{\delta}\in Lip_c(M\setminus \bigcup_{j=1}^kD_{\delta^2}(a_j))$ by
$$\psi_{\delta}(x):=\min_{1\leq j\leq k}\phi_{\delta}(dist(x,a_j)),$$
and observe as in \cite{MooreReam} (or \cite{ChoiSchoen}) that
$$\int |d\psi_{\delta}|^2\leq \frac{C}{|\log\delta|}\to 0$$
as $\delta \to 0$. As a consequence, it's not hard to see that
$$\max\{Q_E(\Psi)(\psi_{\delta}V,\psi_{\delta}V)\mid V\in \mathcal{V},\text{ }\|V\|_{L^2}=1\}\to \max_{0\neq V\in \mathcal{V}}\frac{Q_E(\Psi)(V,V)}{\|V\|_{L^2}^2}<-\beta$$
as $\delta\to 0$. In particular, since linear independence is an open condition, we conclude that the space $\widetilde{\mathcal{V}}=\{\psi_{\delta} V\mid V\in \mathcal{V}\}$ is a $p$-dimensional subspace of $\Gamma(\Psi^*(T\mathbb{S}^n))$, supported away from $\{a_1,\ldots,a_k\}$, and satisfying \eqref{pert.ind.bd}, for $\delta>0$ sufficiently small.

We've now shown that there exists $r_0>0$ and a $p$-dimensional space $\mathcal{V}\subset \Gamma(\Psi^*(T\mathbb{S}^n))$ of Lipschitz variation fields such that
\begin{equation}
\max_{0\neq v\in V}\frac{Q_E(\Psi)(V,V)}{\|V\|_{L^2}^2}<-\beta/2<0
\end{equation}
and
\begin{equation}
\mathrm{supp}(V)\subset M\setminus \bigcup_{j=1}^kD_{r_0}(a_j)
\end{equation}
for every $V\in \mathcal{V}$. Now, by assumption, we know that $\Psi_{\epsilon}\to \Psi$ in $C^2(M\setminus \bigcup_{j=1}^kD_{r_0}(a_j),\mathbb{R}^{n+1})$ as $\epsilon\to 0$. For $\epsilon>0$ and $V\in \mathcal{V}$, we can therefore define
$$V_{\epsilon}:=V-|\Psi_{\epsilon}|^{-2}\langle V,\Psi_{\epsilon}\rangle \Psi_{\epsilon},$$
and observe that $V_{\epsilon}\to V$ in $\mathrm{Lip}(M,\mathbb{R}^{n+1})$ as $\epsilon\to 0$. 
In particular, the space $\mathcal{V}_{\epsilon}:=\{V_{\epsilon}\mid V\in \mathcal{V}\}$ remains $p$-dimensional for $\epsilon>0$ sufficiently small, and since $V_{\epsilon}\perp \Psi_{\epsilon}$ pointwise, direct computation gives
\begin{eqnarray*}
Q_{E_{\epsilon}}(\Psi_{\epsilon})(V_{\epsilon},V_{\epsilon})&=&\int_M(|dV_{\epsilon}|^2-\epsilon^{-2}(1-|\Psi_{\epsilon}|^2)|V_{\epsilon}|^2)\\
\text{(by \eqref{gl.eq})}&=&\int_M(|dV_{\epsilon}|^2-|\Psi_{\epsilon}|^{-2}\langle \Psi_{\epsilon},\Delta \Psi_{\epsilon}\rangle |V_{\epsilon}|^2)\\
&\to &Q_E(\Psi)(V,V)
\end{eqnarray*}
as $\epsilon\to 0$, where in the last line we have used the Lipschitz convergence $V_{\epsilon}\to V$, the $C^2$ convergence $\Psi_{\epsilon}\to \Psi$ away from $\bigcup_{j=1}^kD_{r_0}(a_j)$, and the harmonic map equation 
$$\Delta \Psi=|d\Psi|^2\Psi.$$

Since the convergence $Q_{E_{\epsilon}}(\Psi_{\epsilon})(V_{\epsilon},V_{\epsilon})\to Q_E(\Psi)(V,V)$ is uniform on the compact set $\{V\in \mathcal{V}\mid \|V\|_{L^2}=1\}$, it follows that, for $\epsilon>0$ sufficiently small, $\mathcal{V}_{\epsilon}$ defines a $p$-dimensional space of variations, supported away from $\bigcup_{j=1}^kD_{r_0}(a_j)$, on which $Q_{E_{\epsilon}}(\Psi_{\epsilon})$ is negative definite. This completes the proof of the claim, and therefore of Lemma \ref{bubble.business}.
\end{proof}

By combining the existence result of Proposition \ref{min.max} with the compactness analysis of Lemma \ref{bubble.business}, we can deduce the existence of a harmonic map and a collection of bubbles which together realize the min-max energy $\mathcal{E}_n$. For general min-max constructions of this type, we cannot improve on this conclusion; however, for this special $(n+1)$-parameter construction associated to the first eigenvalue, we can appeal to geometric information to \emph{rule out the occurrence of bubbles}, arriving at the following existence theorem.

\begin{theorem} 
\label{eigenvalue:thm}
Let $(M,g)$ be a surface of positive genus, and let $n>5$. Then there exists a harmonic map $\Psi_n\colon M\to \mathbb{S}^n$ of energy
\begin{equation}
\label{ener.char}
\frac{1}{2}\Lambda_1(M,[g])\leq E(\Psi_n)=\mathcal{E}_n(M,g)\leq V_c(n,M,[g])
\end{equation}
and index
\begin{equation}\label{index.char}
\ind_E(\Psi_n)\leq n+1.
\end{equation}
\end{theorem}

\begin{proof} Combining the results of Proposition \ref{min.max} and Lemma \ref{bubble.business}, for $n\geq 2$, we know that there exist harmonic maps $\Psi_n\colon M\to \mathbb{S}^n$ and $\phi_1,\ldots,\phi_k: \mathbb{S}^2\to \mathbb{S}^n$ such that
\begin{equation}
\label{ener.id.2}
\mathcal{E}_n=E(\Psi_n)+\sum_{j=1}^kE(\phi_j)
\end{equation}
and
\begin{equation}\label{index.bd.2}
n+1\geq \ind_E(\Psi_n)+\sum_{j=1}^k\ind_E(\phi_j).
\end{equation}
The lower and upper bounds on $\mathcal{E}_n$ in \eqref{ener.char} are an immediate consequence of Propositions \ref{lambda.below} and \ref{vc.above}, respectively.

Since $n+1<2(n-2)$ for $n>5$, Proposition~\ref{indEM:prop} implies that  one of the following two possibilities must hold: either there are no nontrivial bubbles $\phi_j$, or there is exactly one bubble $\phi_1$, and the map $\Psi$ is constant.

Assume the latter. Then, by Propostion~\ref{indES2:prop}, $\phi_1$ has to be an equatorial bubble, and by~\eqref{ener.id.2}, one has $\mathcal E_n = 4\pi$. At the same time, combining~\eqref{ener.char} with Theorem~\ref{rigidity:thm} one has
$$
4\pi<\frac{1}{2}\Lambda_1(M,[g])\leq \mathcal{E}_n(M,[g])=4\pi,
$$
which yields a contradiction. Therefore, there are no bubbles and the Theorem is proved.
\end{proof}

\begin{remark} In lower dimensions $3\leq n\leq 5$, we may also rule out bubbles to arrive at the same conclusion whenever we have the energy bound $\mathcal{E}_n(M,[g])<8\pi$.
\end{remark}

\subsection{Stabilization}
\label{stabilization:sec}
To complete the proof of Theorem~\ref{MainTh2:intro}, our goal now is to show that the inequality $\mathcal{E}_n(M,c)\geq \frac{1}{2}\Lambda_1(M,c)$ becomes equality for $n$ sufficiently large; in particular, we wish to show that the maps produced by Theorem \ref{eigenvalue:thm} stabilize in an appropriate sense as $n\to\infty$. As a first step, we observe in the following proposition that the energies $\mathcal{E}_n$ are nonincreasing in $n$.

\begin{proposition}
\label{monotoneEn:prop}
For every $n\geq 2$, $\mathcal E_n(M,[g])\geq \mathcal{E}_{n+1}(M,[g])$.
\end{proposition}

\begin{proof}
Let $\mathbb{B}^{n+1}\subset\mathbb{B}^{n+2}$ be defined by $x_{n+2}=0$. Then 
$$
\mathbb{B}^{n+2}=\{(x',x_{n+2}),\, |x'|^2+|x_{n+2}|^2\leqslant 1 \}.
$$
If $F\in\Gamma_n(M)$, then one constructs $\bar F\in\Gamma_{n+1}(M)$ by the following formula
$$
\bar F_{(x',x_{n+2})} = \left(\sqrt{1-x_{n+2}^2}F_{\frac{x'}{\sqrt{1-x_{n+2}^2}}},x_{n+2}\right)
$$
for $x_{n+2}\ne\pm 1$, and $\bar F_{(0,\pm 1)}=(0,\pm 1)$. Let $\alpha = \sqrt{1-x^2_{n+2}}<1$. Then it easy to see that $\bar F\in \Gamma_{n+1}(M)$ and 
$$
E_{\epsilon}\left(\bar F_{(x',x_{n+2})}\right) = \int_M\frac{\alpha^2}{2}\left|dF_{\frac{x'}{\alpha}}\right|^2 + \frac{\alpha^4}{4\epsilon^2}\left(1 - \left|F_{\frac{x'}{\alpha}}\right|^2\right)^2\leqslant \alpha^2E_{\epsilon}\left(F_{\frac{x'}{\alpha}}\right).
$$
\end{proof}

Now, let $\mathcal C_n$ be the set of all harmonic maps $\Psi_n\colon (M,g)\to \mathbb{S}^n$ satisfying 
$$\ind_E(\Psi_n)\leqslant n+1$$ and 
$$E(\Psi_n) = \mathcal E_n(M,[g]).$$
Theorem \ref{eigenvalue:thm} tells us that $\mathcal{C}_n\neq \varnothing$ for $n>5$.
To prove Theorem~\ref{MainTh2:intro}, our first observation is that if there exists $\Psi_n\in \mathcal C_n$ such that $\ind_S(\Psi_n)=1$, then the inequality 
\begin{equation}
\label{eigenvalue:ineq}
\Lambda_1(M,[g])\leqslant 2\mathcal E_n(M,[g])
\end{equation}
becomes an equality. Indeed, if $\ind_S(\Psi_n)=1$, then 
$$
\Lambda_1(M,[g])\leqslant 2\mathcal E_n(M,[g]) = \lambda_1(M,g_{\Psi_n})\area(M,g_{\Psi_n})\leqslant \Lambda_1(M,[g]).
$$
 In particular, $\mathcal E_m(M,[g]) = \mathcal E_n(M,[g])$ for all $m\geqslant n$. Thus Theorem~\ref{MainTh2:intro} follows from the following proposition.
\begin{proposition}
\label{indS1:prop}
There exists $n\in \mathbb{N}$ and $\Psi\in \mathcal C_n$ such that $\ind_S(\Psi) = 1$.
\end{proposition}
\begin{proof} 

We need the following theorem, which is a slightly stronger version of Propostion~\ref{MainProp:intro} in the introduction.
\begin{theorem}
\label{stabilization:thm}
Let $\Psi_n\colon (M,g)\to\mathbb{S}^{N_n}$ be a collection of harmonic maps to spheres of varying dimensions. If $E(\Psi_n)$ is uniformly bounded, then $\nul_S(\Psi_n)$ is uniformly bounded as well.
\end{theorem}
Indeed, since the components of any harmonic map $\Psi$ are eigenfunctions of $\Delta_{g_\Phi}$ with eigenvalue $2$, a uniform bound on $\nul_S(\Psi)$ implies a uniform bound on the number of linearly independent components of $\Psi$. Assuming $\Psi$ is linearly full, this results in an upper bound on the dimension of the target sphere as in Proposition~\ref{MainProp:intro}.  

We postpone the proof of Theorem~\ref{stabilization:thm} until Section~\ref{stabilization_proof:sec}.  To prove Proposition \ref{indS1:prop}, we apply Theorem~\ref{stabilization:thm} to sequences $\{\Psi_n\}$, $\Psi_n\in \mathcal{C}_n$ to conclude that there exists $N(M,[g])\in \mathbb{N}$ such that 
$$\nul_S(\Psi)\leqslant N+1\text{\hspace{2mm} for all\hspace{2mm} }\Psi\in\bigcup_{n\in \mathbb{N}}\mathcal C_n.$$
It follows that for any $m>N$ and any $\Psi_m\in\mathcal C_m$ the image $\Psi_m$ lies in the $N$-dimensional totally geodesic subsphere of 
$\mathbb{S}^m$. To obtain a contradiction, assume now that the conclusion of Proposition~\ref{indS1:prop} is not valid, i.e. $\ind_S(\Psi_m)>1$. Then by Proposition~\ref{indEindS:prop} one has
$$
\ind_E(\Psi_m)\geqslant (m-N)\ind_S(\Psi_m)\geqslant 2(m-N)>m+1.
$$ 
once $m>2N+1$. As a result, for $m>2N+1$ the space $\mathcal C_m$ is empty, which contradicts Theorem~\ref{eigenvalue:thm}.
\end{proof}

\subsection{Proof of Theorem~\ref{stabilization:thm}}
\label{stabilization_proof:sec}

The proof is based on an analysis of the limiting behavior of the energy densities $|d\Psi_n|_g^2$ of the maps $\Psi_n$, modeled on the bubble tree convergence for harmonic maps to a fixed target (cf. ~\cite{Parker}). The key difference in our case is that the target spaces $\mathbb{S}^{N_n}$ of $\Psi_n$ vary with $n$, so one can not, a priori, expect a compactness result for the maps themselves. Nevertheless, we are able to establish convergence of energy densities in an appropriate "bubble" sense, described in Lemma \ref{bub.lem} below. In what follows, we let
$$\mathcal N_m:=M\sqcup \mathbb{S}^2_1 \sqcup \cdots \sqcup \mathbb{S}^2_m$$
denote the disjoint union of $M$ with $m$ copies of the unit sphere. We endow $\mathcal N_m$ with a metric equal to $g$ on $M$ and the standard metric $g_{\mathbb{S}^2}$ on each sphere component. 

\begin{lemma}\label{bub.lem} Let $\Psi_n\colon (M,g)\to \mathbb{S}^{N_n}$ be a sequence of harmonic maps with $E(\Psi_n)\leq K$. After passing to a subsequence, there exists $m\in \{0\}\cup \mathbb{N}$, a finite collection of points $p_1,\ldots,p_k\in \mathcal N_m$, and a sequence of neighborhoods
$$\{p_1,\ldots,p_k\}\subset \mathcal{B}_n\subset \mathcal N_m$$
converging in the Hausdorff sense to $\{p_1,\ldots,p_k\}$, such that on the complement of $\mathcal{B}_n$, there exist surjective conformal maps
$$\Phi_n:\mathcal N_m\setminus \mathcal{B}_n\to M$$
whose restriction to each component $M\setminus \mathcal{B}_n, \mathbb{S}^2_1\setminus \mathcal{B}_n, \ldots, \mathbb{S}^2_m\setminus \mathcal{B}_n$ is a diffeomorphism onto its image. Moreover, these images have disjoint interiors, and there exists $p>1$ and $\rho \in L^p(\mathcal N_m)$ such that
$$\rho_n:=|d(\Psi_n\circ \Phi_n)|^2\boldsymbol{1}_{\mathcal N_m\setminus \mathcal{B}_n}\to \rho\text{ in }L^p(\mathcal N_m).$$
\end{lemma}

The key point in Lemma \ref{bub.lem} is the $L^p$ convergence of the conformally rescaled energy densities to a limit density $\rho\in L^p(\mathcal N_m)$. In what follows, we will often write $\lambda_k(M,\rho):=\lambda_k(M,[g],\rho dv_g)$ for $\rho\in L^1(M)$.
With Lemma~\ref{bub.lem} in place, we can establish the following lower-semicontinuity result for the eigenvalues $\lambda_k(M,|d\Psi_n|^2)$ of the energy density measures, from which Theorem \ref{stabilization:thm} will follow. 

\begin{proposition}\label{eigen.drop} In the setting of the Lemma \ref{bub.lem}, we have
$$
\liminf_{n\to\infty}\lambda_k(M,|d\Psi_n|^2)\geq \lambda_k(\mathcal N_m,\rho).
$$
\end{proposition}

\begin{proof} To begin, consider a normalized collection $\frac{1}{\sqrt{2E(\Psi_n)}}=\phi_{n,0},\ldots,\phi_{n,k}\in C^{\infty}(M)$ of first $k+1$ eigenfunctions for the measure $|d\Psi_n|_g^2dv_g$ on $M$, satisfying
\begin{eqnarray*}
\int\limits_M \langle d\phi_{n,i},d\phi_{n,j}\rangle dv_g&=&\lambda_i(M,|d\Psi_n|^2)\int\limits_M\phi_{n,i}\phi_{n,j}|d\Psi_n|_g^2dv_g\\
&=&\lambda_i(M,|d\Psi_n|_g^2)\cdot \delta_{ij},
\end{eqnarray*} 
see~\eqref{eigenfunctions_measures:def}.
Now, consider the functions $\psi_{n,i}:=\phi_{n,i}\circ \Phi_n\in C^{\infty}(\mathcal N_m\setminus \mathcal{B}_n)$ on $\mathcal N_m\setminus \mathcal{B}_n$ given by composition with the conformal maps $\Phi_n$. Since $\Phi_n$ is a conformal diffeomorphism away from the boundary $\partial (\mathcal N_m\setminus \mathcal{B}_n)$, it is then clear that $\psi_{n,i}\in W^{1,2}(\mathcal N_m\setminus \mathcal{B}_n)$, with
$$\int\limits_{\mathcal N_m\setminus \mathcal{B}_n}|d\psi_{n,i}|^2=\int\limits_M|d\phi_{n,i}|^2\leq \lambda_k(M,|d\Psi_n|^2)$$
and (since $supp(\rho_n)\subset \mathcal{N}_m\setminus \mathcal{B}_n$)
\begin{equation}
\label{ref:ineq1}
\int\limits_{\mathcal N_m}\psi_{n,i}\psi_{n,j}\rho_n =\int\limits_{\mathcal{N}_m\setminus \mathcal{B}_n}[(\phi_{n,i}\phi_{n,j})\circ \Phi_n ]|d(\Psi_n\circ \Phi_n)|^2=\int\limits_M\phi_{n,i}\phi_{n,j}|d\Psi_n|^2=\delta_{ij}.
\end{equation}

Next, extend the functions $\psi_{n,i}\in C^{\infty}(\mathcal N_m\setminus \mathcal{B}_n)$ \emph{harmonically} to $\mathcal{B}_n$, to obtain functions
$$\overline{\psi}_{n,i}\in W^{1,2}(\mathcal N_m)$$
agreeing with $\psi_{n,i}$ on $\mathcal N_m\setminus \mathcal{B}_n$. By~\cite[Example 1, p. 40]{RT} these extensions satisfy
\begin{equation}\label{psi.ener.bd}
\int\limits_{\mathcal N_m}|d\overline{\psi}_{n,i}|^2\leq C \int\limits_{\mathcal N_m\setminus \mathcal{B}_n}|d\psi_{n,i}|^2.
\end{equation}

Now, while \eqref{psi.ener.bd} provides a uniform bound on the Dirichlet energies of $\overline{\psi}_{n,i}\in W^{1,2}(\mathcal N_m)$, it remains to show that these functions are bounded in $L^2(\mathcal N_m)$ as well, to extract a subsequence. But this follows in a straightforward way from the following theorem.

\begin{theorem}[\cite{AH} Lemma 8.3.1] 
\label{Poincare:thm}
Let $(M,g)$ be a Riemannian manifold. Then there exists a constant $C>0$ such that for all $L\in W^{-1,2}(M)$ with $L(1) = 1$ one has 
\begin{equation}
||u- L(u)||_{L^2(M)} \leqslant C||L||_{W^{-1,2}(M)} \left(\,\int\limits_M |\nabla u|^2_g\,dv_g\right)^{1/2}
\end{equation}
for all $u\in W^{1,2}(M)$.
\end{theorem}

We apply the theorem for $L_n(\psi) = \frac{1}{\|\rho_n\|_{L^1}}\int\limits_{\mathcal N_m} \psi\rho_n\,dv_g$. Since $L^p(\mathcal N_m)$ embeds into $W^{-1,2}(\mathcal N_m)$ for $p>1$ (by the dual form of the Sobolev embedding theorem), we know that
$$\|\rho_n\|_{W^{-1,2}(\mathcal N_m)}\leq C_p\|\rho_n\|_{L^p(\mathcal N_m)}\leq C$$
for some constant $C$ independent of $n$, and by Theorem~\ref{Poincare:thm} one has
$$
\int\limits_{\mathcal N_m} \left(\overline{\psi}_{n.i} - \frac{1}{||\rho_n||_{L^1}}\int\overline{\psi}_{n,i}\rho_n\,dv_g\right)^2\,dv_g\leqslant C\int\limits_M |\nabla \overline{\psi}_{n,i}|^2\leqslant C \lambda_k(M,|d\Psi_n|^2).
$$
In particular, since $i>0$, by \eqref{ref:ineq1} one has
$$
\int\limits_{\mathcal N_m}\overline{\psi}_{n,i}\rho_n\,dv_g = \sqrt{2E(\Psi_n)}\int\limits_{\mathcal N_m}\overline{\psi}_{n,i}\overline{\psi}_{n,0}\rho_n\,dv_g=0.
$$
It follows that the functions $\overline{\psi}_{n,1},\ldots, \overline{\psi}_{n,k}$ are bounded in $W^{1,2}(\mathcal N_m)$ by $C\lambda_k(M,|d\Psi_n|^2)$. If $\liminf_{n\to\infty}\lambda_k(M,|d\Psi_n|^2) = \infty$, then the statement of Proposition~\ref{eigen.drop} is obvious.
Otherwise, for each $i=1,\ldots,k$ we can extract a subsequence (unrelabelled) such that 
$$\overline{\psi}_{n,i}\to \psi_i\in W^{1,2}(\mathcal N_m)$$
weakly in $W^{1,2}(\mathcal N_m)$ and strongly in $L^s(\mathcal N_m)$ for every $s\in [1,\infty)$. We can assume the same for $i=0$ since $\overline{\psi}_{n,0}$ are constant functions. In particular, since $\rho_n\to \rho$ in $L^p(\mathcal N_m)$ and $\overline{\psi}_{n,i}\to \psi_i$ strongly in $L^{2p'}(\mathcal N_m)$, where $p'$ is the H\"{o}lder conjugate of $p$, it's clear that
$$\int\limits_{\mathcal N_m}\psi_i\psi_j \rho=\lim_{n\to\infty}\int\limits_{\mathcal N_m}\psi_{n,i}\psi_{n,j}\rho_n=\delta_{ij}.$$
Moreover, since $\Delta_g\overline\psi_{n,i} = \lambda_i(M,|d\Psi_n|^2)\psi_{n,i}\rho_n$ on $\mathcal N_m\setminus \mathcal B_n$, then for any $\eta\in C^{\infty}_c(\mathcal N_m\setminus \{p_1,\ldots,p_k\})$ supported away from the points $\{p_1,\ldots,p_k\}$, the weak convergence $\overline{\psi}_{n,i}\to \psi_i$ of the eigenfunctions in $W^{1,2}(\mathcal N_m)$ easily gives
\begin{equation}\label{eigen.lim}
\int\limits_{\mathcal N_m}\langle d\psi_i,d\eta\rangle=[\lim_{n\to\infty}\lambda_i(M,|d\Psi_n|^2)]\int\limits_{\mathcal N_m}\psi_i\eta \rho,
\end{equation}
and since the set $\{p_1,\ldots, p_k\}$ has capacity zero, it follows that \eqref{eigen.lim} holds for any $\eta\in W^{1,2}(\mathcal N_m)$. In particular, taking $\eta=\psi_i$, we see that $\psi_0,\ldots,\psi_k$ define an orthonormal collection of functions in $L^2(\mathcal N_m,\rho)$ with
$$\int\limits_{\mathcal N_m}|d\psi_i|^2\leq \lim_{n\to\infty}\lambda_k(M,|d\Psi_n|^2),$$
and it follows from the definition of $\lambda_k(M,\rho)$ that
$$\lambda_k(M,\rho)\leq \lim_{n\to\infty}\lambda_k(M,|d\Psi_n|^2),$$
as desired.
\end{proof}

We can now complete the proof of Theorem \ref{stabilization:thm} in a few lines.

\begin{proof}[Proof of Theorem \ref{stabilization:thm}]

Let $\Psi_n: (M^2,g)\to \mathbb{S}^{N_n}$ be a sequence of harmonic maps, and suppose that $E(\Psi_n)$ is uniformly bounded. To obtain a contradiction, suppose that
$$\nul_S(\Psi_n)\to \infty\text{ as }n\to\infty.$$
By definition of $\nul_S$, there exist $\nul_S(\Psi_n)$ linearly independent eigenfunctions for $|d\Psi_n|_g^2$ corresponding to the eigenvalue $1$, so in particular,
$$\lambda_{\nul_S(\Psi_n)}(M,|d\Psi_n|_g^2)\leq 1.$$
Thus, if $\nul_S(\Psi_n)\to \infty$, then we have
$$\lim_{n\to\infty}\lambda_k(M,|d\Psi_n|_g^2)\leq 1$$
for every $k\in \mathbb{N}$. 

Passing to a further subsequence if necessary, it follows from Lemma \ref{bub.lem} and Proposition \ref{eigen.drop} that there exists a function $\rho\in L^p(\mathcal N_m)$ for some $m\in\{0\}\cup \mathbb{N}$ and $p>1$ such that 
$$\int_{\mathcal N_m}\rho=\lim_{n\to\infty}E(\Psi_n)>0$$
and
$$\lambda_k(\mathcal N_m,\rho)\leq \lim_{n\to\infty}\lambda_k(M,|d\Psi_n|_g^2)\leq 1$$
for every $k\in \mathbb{N}$. On the other hand, it follows from Propositions~\ref{ef:prop} and~\ref{ef_Lp:prop} that the sequence 
$$\lambda_k(\mathcal N_m,\rho)\to \infty$$
as $k\to\infty$, giving us the desired contradiction.
\end{proof}

In the following subsection, we complete the argument by proving Lemma \ref{bub.lem}, establishing the $L^p$-compactness of the energy densities after conformal rescalings.  

\subsection{Proof of Lemma \ref{bub.lem}}

The starting point for our bubbling analysis is the following lemma, in which we observe that the constants in the standard small-energy regularity theorem for harmonic maps to $\mathbb{S}^n$ are independent of the dimension $n$ of the target sphere.

\begin{lemma}
\label{regularity:lemma}
Let $\Psi: M\to \mathbb{S}^n$ be a harmonic map. There exists $\varepsilon_0>0$ and $r_0>0$ independent of $n$ such that, for all $x\in M$ and $r<r_0$, if 
$$
\int\limits_{B_{2r}(x)}|d\Psi|^2_g\,dv_g<\varepsilon_0,
$$
then on $B_r(x)$ one has 
\begin{equation}\label{c0.est}
r^2\sup_{B_r(x)}|d\Psi|^2_g\leqslant C\varepsilon_0.
\end{equation}

\end{lemma}
\begin{proof}
For a fixed target manifold, this is simply the classical $\epsilon$-regularity theorem--see e.g.~\cite{SU}. The key claim we make here is that for sphere-valued harmonic maps, the constants $\varepsilon_0$, $r_0$, and $C$ will not depend on the dimension $n$ of the target sphere $\mathbb{S}^n$. To see that, one could, for example, examine the proof in~\cite[pp. 149 -- 151]{CM}, noting that the energy density $|d\Psi|^2$ of a sphere-valued harmonic map $\Psi:M\to \mathbb{S}^n$ satisfies the same Bochner identity
\begin{equation}\label{boch}
\Delta_g \frac{1}{2}|d\Psi|_g^2=-|\mathrm{Hess}(\Psi)|^2-K_M|d\Psi|^2+|d\Psi|^4
\end{equation}
for any $n\in \mathbb{N}$.
\end{proof}

Lemma~\ref{regularity:lemma} yields the following preliminary compactness result for the energy densities.
\begin{lemma}
\label{convergence:lemma}
Let $\Psi_n:M\to \mathbb{S}^{N_n}$ be a sequence of harmonic maps with $E(\Psi_n)\leq K$. After passing to a subsequence, there exists a non-negative function $e_\infty\in L^1(M)$, a collection of points $P=\{p_1,\ldots,p_l\}\in M$ and weights $w_i\geqslant \varepsilon_0$, $i=1,\ldots l$ such that
\begin{equation}
\label{convergence:formula1}
|d\Psi_n|^2_g\,dv_g\rightharpoonup^*e_\infty\,dv_g + \sum_{i=1}^lw_i\delta_{p_i},
\end{equation}
and 
\begin{equation}
\label{convergence:formula2}
|d\Psi_n|^2\to e_{\infty}\text{ strongly in }L^q(\Omega)
\end{equation}
for any $q\in [1,\infty)$ and any domain $\Omega \Subset M\setminus P$ supported away from $P$. 
\end{lemma}
\begin{proof}
With Lemma \ref{regularity:lemma} in place, the proof of~\eqref{convergence:formula1} is identical to that of analogous energy concentration results for harmonic maps to a fixed target (cf., e.g., Lemma 1.2 in~\cite{Parker}). 

To prove~\eqref{convergence:formula2} let $\Omega \Subset M\setminus P$, then by Lemma~\ref{regularity:lemma} there exists $C_\Omega$ such that $|d\Psi_n|^2_g\leqslant C_\Omega$. Moreover, formula~\eqref{convergence:formula1} implies that $|d\Psi_n|^2_g\to e_\infty$ in $L^1(\Omega)$ and, in particular, $|e_\infty|\leqslant C_\Omega$, $dv_g$-a.e. Therefore, for any $q\in [1,\infty)$ one has
$$
\int\limits_\Omega \left||d\Psi_n|^2_g - e_\infty\right|^q\,dv_g\leqslant (2C_\Omega)^{q-1}\int\limits_\Omega \left||d\Psi_n|^2_g - e_\infty\right|\,dv_g\to 0.
$$

\end{proof}

Suppose now that $\Psi_n: M\to \mathbb{S}^{N_n}$ is a sequence of harmonic maps with $E(\Psi_n)\leq K$ satisfying the conclusions of Lemma \ref{convergence:lemma}. In order to better understand the behavior of energy densities in the neighbourhood of bubble points $p_i$, we rescale the measures and repeat the procedure. To this end we fix a bubble point $p_i$ and omit the subscript $i$ for convenience.

 In the following we use the notation $\delta_n\ll \varepsilon_n$ whenever $\frac{\delta_n}{\varepsilon_n}\to 0$ as 
$n\to\infty$. In particular $\delta_n\ll 1$ means that $\delta_n\to 0$.  

There exists a neighborhood $U$ of the point $p$ such that the metric $g$ on $U$ is conformally flat, $g = f_pg_p$, where $g_p$ is flat. In the following, when we are working in the neighborhood of $p$, the neighborhood is always a subset of $U$ and the distance is measured with respect to $g_p$. Let $r_1>0$ be such that $B_{2r_1}(p)\subset U$ contains no bubble points other than $p$.

We define now a sequence of scales $\varepsilon_n>0$ by setting
\begin{equation}\label{cut.scale}
\varepsilon_n^2:=\inf\{r>0\mid \||d\Psi_n|^2-e_{\infty}\|_{L^1(B_{2r_1}(p)\setminus B_r(p))}<r\}.
\end{equation}
Since $|d\Psi_n|^2\to e_{\infty}$ in $L^1(B_{2r_1}(p)\setminus B_r(p))$ for any fixed $r>0$ by Lemma \ref{convergence:lemma}, it follows that
$$\varepsilon_n \ll 1.$$
In particular, by the definition of $\varepsilon_n$, we have
 \begin{equation}
 \label{tree:eq0}
 \||d\Psi_n|^2_g-e_\infty\|_{L^1(B_{2r_1}(p)\setminus B_{\varepsilon_n^2}(p))}\leq \varepsilon_n^2 \ll 1,
 \end{equation}
In what follows, without loss of generality, we identify $B_{2\varepsilon_n}(p)$ with a ball $B_{2\varepsilon_n}(0)$ in $\mathbb{R}^2$. 

Now, fix a normalization constant 
$$0<C_R<\varepsilon_0,$$
and define a function
$$\alpha: B_{\varepsilon_n}(p)\to [0,2\varepsilon_n]$$
implicitly by requiring that
$$\int\limits_{B_{\varepsilon_n}(p)\setminus B_{\alpha(x)}(x)}|d\Psi_n|_g^2dv_g=C_R.$$
It's easy to see that $\alpha$ is continuous, and therefore achieves a minimum in $B_{\varepsilon_n}(p)$; we then define
\begin{equation}\label{alpha.def}
\alpha_n:=\min_{x\in B_{\varepsilon_n}(p)}\alpha(x),
\end{equation}
and let $c_n\in B_{\varepsilon_n}(p)$ be a point such that
\begin{equation}\label{cn.def}
\alpha(c_n)=\alpha_n>0,
\end{equation}
where the inequality follows from the fact that $C_R<\varepsilon_0\leqslant w_p$.
A similar choice of $\alpha_n$, $c_n$ in the context of bubbling construction appears in~\cite[p.188]{CM}.
We record the following useful properties of $\alpha_n$ and $c_n$, necessary for the rescaling procedure.

\begin{lemma}
\label{tree:lemma1}
As $n\to\infty$, for a sequence of points $c_n\in B_{\varepsilon_n}(p)$ satisfying \eqref{cn.def}, we have
$$\lim_{n\to\infty}\frac{dist(c_n,p)}{\varepsilon_n}=\lim_{n\to\infty}\frac{\alpha_n}{\varepsilon_n}=0,$$
and
 \begin{equation}
 \label{tree:eq2}
  \int\limits_{B_{\varepsilon_n(c_n)}}|d\Psi_n|^2_g\,dv_g = w_p+o(1)
 \end{equation}
as $n\to\infty$.
 We also have
 \begin{equation}
 \label{tree:eq1}
 |d\Psi_n|^2\boldsymbol{1}_{B_{r_1}(p)\setminus B_{\varepsilon_n}(c_n)}\to e_\infty \text{ in }L^1(M)
 \end{equation}
as $n\to\infty$,
 and
 $$
 B_{2^{-1}\varepsilon_n}(p)\subset B_{\varepsilon_n}(c_n)\subset B_{2\varepsilon_n}(p).
 $$
\end{lemma}

\begin{proof} First, note that by definition \eqref{cut.scale} of $\varepsilon_n$, and since $e_{\infty}\in L^1(M)$, we have
\begin{eqnarray*}
\lim_{n\to\infty}\int\limits_{B_{\varepsilon_n}(p)\setminus B_{\varepsilon_n^2}(p)}|d\Psi_n|_g^2dv_g&\leq & \lim_{n\to\infty}[\varepsilon_n^2+\int\limits_{B_{\varepsilon_n}(p)}e_{\infty}]\\
&=&0;
\end{eqnarray*}
in particular,
\begin{equation}\label{off.e2.small}
\int\limits_{B_{\varepsilon_n}(p)\setminus B_{\varepsilon_n^2}(p)}|d\Psi_n|_g^2dv_g<C_R
\end{equation}
for $n$ sufficiently large, so it follows that
$$\alpha_n\leq \alpha(p)\leq \varepsilon_n^2\ll \varepsilon_n$$
as $n\to\infty$.

Next, it also follows from \eqref{off.e2.small} and the definition of $\alpha_n=\alpha(c_n)$ that
$$B_{\alpha_n}(c_n)\cap B_{\varepsilon_n^2}(p)\neq \varnothing$$
as $n\to\infty$. In particular, we see that
$$dist(c_n,p)\leq \varepsilon_n^2+\alpha_n\leq 2\varepsilon_n^2\ll \varepsilon_n$$
as $n\to\infty$.

Having shown that $\alpha_n\ll\varepsilon_n$ and $dist(c_n,p)\ll \varepsilon_n$ as $n\to\infty$, the remaining statements follow easily from \eqref{tree:eq0}.

\end{proof}

In addition to the properties outlined in Lemma \ref{tree:lemma1}, it will be useful to note the following: for any sequence $x_n\in B_{2^{-1}\varepsilon_n}(p)$, we claim that
\begin{equation}\label{bubble.drop}
\lim_{n\to\infty}\int\limits_{B_{\alpha_n}(x_n)}|d\Psi_n|^2_gdv_g\leq w_p-C_R.
\end{equation}
Indeed, this follows easily from the fact that $\alpha_n\leq \alpha(x_n)$ by definition, since
\begin{eqnarray*}
\int\limits_{B_{\alpha_n}(x_n)}|d\Psi_n|_g^2dv_g&\leq &\int\limits_{B_{\alpha(x_n)}(x_n)}|d\Psi_n|_g^2dv_g\\
&=&\int\limits_{B_{\varepsilon_n}(p)}|d\Psi_n|_g^2dv_g-C_R\\
&\to &w_p-C_R\text{ as }n\to\infty.
\end{eqnarray*}

With these preparations in place, Lemma~\ref{tree:lemma1} allows us to do the following rescaling. Let $\pi\colon\mathbb{R}^2\to\mathbb{S}^2$ be the inverse stereographic projection. We consider the conformal map 
$$R_n \colon B_{\varepsilon_n}(p)\to\mathbb{S}^2,\text{\hspace{3mm}} R_n(x)=\pi(\alpha_n^{-1}(x-c_n)),$$ 
and denote the image by $\Omega_n(p)\subset \mathbb{S}^2$. 
By Lemma~\ref{tree:lemma1} $\alpha_n\ll\varepsilon_n$, and it follows that domains $\Omega_n(p)$ exhaust $\mathbb{S}^2\setminus\{S\}$, where $S$ is the south pole. 
Thus, for any compact $K\Subset \mathbb{S}^2\setminus\{S\}$ one has $K\Subset \Omega_n(p)$ for large enough $n$. 

Denote by $\widetilde\Psi_n: \Omega_n(p)\subset \mathbb{S}^2\to \mathbb{S}^{N_n}$ the compositions
$$\widetilde\Psi_n:=\Psi_n\circ R_n^{-1}.$$
By the conformal invariance of harmonic maps, $\widetilde \Psi_n$ are harmonic, with energies 
$$E(\widetilde\Psi_n|_{\Omega_n})=w_p+o(1).$$
Thus, we can apply Lemma~\ref{convergence:lemma} to restrictions  $\widetilde \Psi_n|_K$ and pass to a diagonal subsequence over a compact exhaustion $\{K_j\}$ of $\mathbb{S}^2\setminus \{S\}$ to arrive at 
\begin{equation}
\label{2ndarybubble:limit}
|d\widetilde \Psi_n|^2_g\,dv_g\boldsymbol{1}_{\Omega_n}\rightharpoonup^*\widetilde \nu = \widetilde e_\infty\,dv_g + \sum_{i=1}^{\widetilde l}w_{\widetilde p_i}\delta_{\widetilde p_i} + \tau_S\delta_S,
\end{equation}
where $g$ is the round metric on $\mathbb{S}^2$, $\tau_S\leqslant C_R$ and $\varepsilon_0\leqslant w_{\widetilde p_i}$. 
The points $\widetilde p_i\in \mathbb{S}^2\setminus S$ are called secondary bubbles. 

We can repeat this rescaling process at secondary bubbles $\widetilde p_i$ to arrive at new bubbles, and iterate. We observe now that this process terminates after finitely many steps.
\begin{lemma}
The mass $w_{\widetilde p_i}$ of each secondary bubble $\widetilde p_i$ is at most $w_p-C_R$. As a result there are no more bubbles after $\left\lfloor\frac{w}{C_R}\right\rfloor$ steps.
\end{lemma}
\begin{proof}

For any domain $K_a\subset \mathbb{S}^2$ given by the image under stereographic projection
$$K_a=\pi(B_1(a))$$
of a unit disk $B_1(a)$ in $\mathbb{R}^2$, it follows from \eqref{bubble.drop} and the definition of the maps $\tilde{\Psi}_n$ that
\begin{eqnarray*}
\widetilde\nu(K_a)&=&\lim_{n\to\infty}\int\limits_{\pi(B_1(a))}|d\widetilde\Psi_n|_g^2dv_g\\
&=&\lim_{n\to\infty}\int\limits_{R_n(B_{\alpha_n}(x_n))}|d\widetilde\Psi_n|_g^2dv_g\\
&\leq & w_p-C_R.
\end{eqnarray*}

In particular, every secondary bubble $\tilde{p}_i\in \mathbb{S}^2\setminus \{S\}$ has an open neighborhood $U$ on which $\tilde{\nu}(U)\leq w_p-C_R$, and the lemma follows.

\end{proof}

Let $m\in \mathbb{N}$ denote the total number of bubbles arising from this process, and recall that $\mathcal N_m$ denotes the disjoint union
$$\mathcal N_m:=M\sqcup \mathbb{S}^2_1\sqcup \cdots \sqcup \mathbb{S}^2_m$$
of $M$ with $m$ copies of the unit $2$-sphere. We are now in a position to define the sets $\mathcal{B}_n\subset \mathcal N_m$ and maps $\Phi_n: \mathcal N_m\setminus \mathcal{B}_n\to M$ of Lemma \ref{bub.lem} as follows. For each bubble point $p_1,\ldots, p_l\in M$, we define the scales $\varepsilon_n^i=\varepsilon_n(p_i)$ and $\alpha_n^i=\alpha_n(p_i)$ by \eqref{cut.scale} and \eqref{alpha.def} as before, and choose $c_n^i=c_n(p_i)\in B_{\varepsilon_n(p_i)}(p_i)$ satisfying \eqref{cn.def}; we then set
$$\mathcal{B}_n\cap M:=\bigcup_{i=1}^{\ell}B_{\sqrt{\alpha_n^i}}(c_n^i),$$
while defining
$$\Phi_n|_{M\setminus \mathcal{B}_n}:=\mathrm{Id}:M\setminus \mathcal{B}_n\to M.$$
Next, on the sphere $\mathbb{S}^2_{p_i}\subset \mathcal N_m$ associated to the bubble at $p_i$, consider the sets
$$\Xi_n(p_i)=R_n(B_{\sqrt{\alpha_n^i}}(c_n(p_i)))$$
exhausting $\mathbb{S}^2\setminus \{S\}$, and at each secondary bubble $\tilde{p}_1,\ldots, \tilde{p}_{\tilde{l}}\in \mathbb{S}^2\setminus \{S\}$, let $\alpha_n(\tilde{p}_i)$ and $c_n(\tilde{p_i})$ be given by \eqref{cut.scale} and \eqref{cn.def}. We then define
$$\mathcal{B}_n\cap \mathbb{S}^2_{p_i}:=\left(\mathbb{S}^2\setminus\Xi_n(p_i)\right)\cup \bigcup_{j=1}^{\tilde{l}}B_{\sqrt{\alpha_n(\tilde{p_j})}}(c_n(\tilde{p_j})),$$
and set
$$\Phi_n|_{\mathbb{S}^2_{p_i}\setminus \mathcal{B}_n}:=R_n^{-1}:\mathbb{S}^2_{p_i}\setminus \mathcal{B}_n\to M.$$
The definition of $\mathcal{B}_n$ and $\Phi_n$ on the spheres associated to secondary bubbles is analogous, with $\Phi_n$ now given by the composition
$$\Phi_n:=R_{n,p_i}^{-1}\circ R_{n,\widetilde p_i}^{-1},$$
and we carry on this way to extend the definition of $\mathcal{B}_n$ and $\Phi_n$ to all of $\mathcal N_m$.

Now, we define the limiting density function $\rho\in L^1(\mathcal N_m)$ by setting
$$\rho|_M:=e_{\infty}\text{ and }\rho|_{\mathbb{S}^2_i}=\widetilde{e}_{\infty},$$
where $e_{\infty}$ and $\widetilde{e}_{\infty}$ are the absolutely continuous parts of the limiting energy measures in Lemma \ref{convergence:lemma} and \eqref{2ndarybubble:limit}. To prove Lemma \ref{bub.lem}, we need to check that the restricted energy densities
$$\rho_n:=|d(\Psi_n\circ \Phi_n)|^2\boldsymbol{1}_{\mathcal N_m\setminus \mathcal{B}_n}$$
converge strongly to $\rho$ in $L^q(\mathcal N_m)$ for some $q>1$. Thus, it remains to check the following proposition.

\begin{proposition}\label{lq.cvg} In the settings of Lemma \ref{convergence:lemma} and \eqref{2ndarybubble:limit}, there exists $q>1$ for which we have the strong convergence
$$|d\Psi_n|_g^2\boldsymbol{1}_{M\setminus \bigcup_{i=1}^lB_{\sqrt{\alpha_n^i}}(c_n^i)}\to e_{\infty}\text{ in }L^q(M)$$
and
$$|d\widetilde{\Psi}_n|_g^2\boldsymbol{1}_{\Xi_n\setminus \bigcup_{j=1}^{\tilde{l}}B_{\sqrt{\alpha_n(\widetilde p_j)}}(c_n(\widetilde{p_j}))}\to \widetilde{e}_{\infty}\text{ in }L^q(\mathbb{S}^2).$$
\end{proposition}

The proof rests largely on the following estimate, whose proof we postpone to the end of the section. 

\begin{lemma}
\label{limit_measure:lemma}
At a bubble point $p\in M$ (or secondary bubble $\tilde{p}\in \mathbb{S}^2$), for every $r>\sqrt{\alpha_n(p)}$, denote by $A_{r,n}(c_n(p))$ the annulus
$$A_{r,n}(c_n):=B_r(c_n)\setminus B_{r^{-1}\alpha_n}(c_n).$$
For an appropriate choice of normalization constant $C_R\in (0,\varepsilon_0)$ in the definition of $\alpha_n$, there exist $\rho_0>0$, $\sigma>0$, and $C<\infty$ such that
$$
\lim_{n\to\infty}\sup_{\rho_0>r>\sqrt{\alpha_n}}r^{-\sigma}\int\limits_{A_{r,n}}|d\Psi_n|_g^2\,dv_g< C.
$$
\end{lemma}

With this estimate in place, we prove Proposition \eqref{lq.cvg} as follows.

\begin{proof}[Proof of Proposition \ref{lq.cvg}]

By Lemma \ref{convergence:lemma}, we know already that $\rho_n\to \rho$ strongly in $L^q(K)$ for any $q<\infty$ if $K\Subset \mathcal N_m\setminus \{p_1,\ldots,p_k\}$ is a compact set away from the singular points $\{p_i\}$. Thus, to complete the proof of the lemma, it is enough to show that for any bubble point $p\in M$ (or secondary bubble $\tilde{p}\in \mathbb{S}^2$), there exists a neighborhood $V_p\subset M$ of $p$ and a neighborhood $W_S\subset \mathbb{S}^2_p$ of the south pole in $\mathbb{S}^2$ such that
\begin{equation}\label{nearbub.cvg}
e_n:=|d\Psi_n|_g^2\boldsymbol{1}_{M\setminus B_{\sqrt{\alpha_n}}(c_n)}\to e_{\infty}\text{ in }L^q(V_p)
\end{equation}
and
\begin{equation}\label{nearneck.cvg}
\widetilde{e}_n:=|d\widetilde\Psi_n|_g^2\boldsymbol{1}_{\Xi_n}\to \widetilde e_{\infty}\text{ in }L^q(W_S)
\end{equation}
for some $q>1$. 

We begin with \eqref{nearbub.cvg}. Recall that, by definition \eqref{alpha.def} of the scales $\alpha_n$, we have
$$\int\limits_{B_{\varepsilon_n}(p)\setminus B_{\alpha_n}(c_n)}|d\Psi_n|_g^2dv_g\leq C_R<\varepsilon_0,$$
while it follows from \eqref{tree:eq0} that
\begin{equation}\label{l1.obv}
\lim_{n\to\infty}\int\limits_{B_r(p)\setminus B_{\varepsilon_n}(p)}|d\Psi_n|_g^2dv_g=\int\limits_{B_r(p)}e_{\infty}dv_g
\end{equation}
for any fixed $0<r<2r_1$. In particular, since we can make the right hand side of \eqref{l1.obv} as small as we like by taking $r$ sufficiently small, it follows that we can choose some $r_2\in (0,r_1)$ such that
\begin{equation}\label{small.e.ann}
\int\limits_{B_{2r_2}(p)\setminus B_{\alpha_n}(c_n)}|d\Psi_n|_g^2dv_g<\varepsilon_0.
\end{equation}
As a consequence, for $n$ sufficiently large, and any $x\in B_{r_2}(c_n)\setminus B_{\sqrt{\alpha_n}}(c_n)$, writing
$$d_{c_n}(x):=\mathrm{dist}(x,c_n),$$
we see that $B_{d_{c_n}(x)/2}(x)\subset B_{2r_2}(p)\setminus B_{\alpha_n}(c_n),$ so in particular,
$$\int\limits_{B_{d_{c_n}(x)/2}(x)}|d\Psi_n|_g^2dv_g<\varepsilon_0.$$
Thus, we can apply Lemma \ref{regularity:lemma} in the balls $B_{d_{c_n}(x)/2}(x)$ to conclude that
$$d_{c_n}(x)^2|d\Psi_n|_g^2(x)\leq C\text{ for all }x\in B_{r_2}(p)\setminus B_{\sqrt{\alpha_n}}(c_n);$$
i.e.,
\begin{equation}\label{ann.linfty.1}
d_{c_n}(x)^2e_n(x)\leq C\text{ on }B_{r_2}(c_n).
\end{equation}

Next, note that by Lemma \ref{limit_measure:lemma}, for every $r>\sqrt{\alpha_n}$ we have the estimate
\begin{eqnarray*}
\int\limits_{B_r(c_n)}e_n&=&\int\limits_{B_r(c_n)\setminus B_{\sqrt{\alpha_n}(c_n)}}|d\Psi_n|_g^2dv_g\\
&\leq &\int\limits_{B_r(c_n)\setminus B_{r^{-1}\alpha_n}(c_n)}|d\Psi_n|_g^2dv_g\\
&<& Cr^{\sigma}
\end{eqnarray*}
for some fixed $C<\infty$ and $\sigma>0$. In particular since $\int\limits_{B_{\sqrt{\alpha_n}(c_n)}}e_n=0$ by definition of $e_n$, it follows that
\begin{equation}\label{sigma.dens.bd}
\int\limits_{B_r(c_n)}e_n< Cr^{\sigma}
\end{equation}
for all $r>0$. 

We claim now that \eqref{ann.linfty.1} and \eqref{sigma.dens.bd} together imply uniform bounds for $\|e_n\|_{L^q(B_{r_2/2}(p))}$ for every $q<1+\frac{\sigma}{2}$. To see this, simply note that $B_{r_2/2}(p)\subset B_{r_2}(c_n)$ for $n$ sufficiently large, and estimate
\begin{eqnarray*}
\int\limits_{B_{r_2}(c_n)}e_n^q&=&\sum_{j=0}^{\infty}\int\limits_{[B_{2^{-j}r_2}\setminus B_{2^{-j-1}r_2}]}e_n^{q-1}\cdot e_n\\
\text{ by \eqref{ann.linfty.1} }&\leq &\sum_{j=0}^{\infty}(Cr_2^{-1}2^{j+1})^{2(q-1)}\int\limits_{[B_{2^{-j}r_2}\setminus B_{2^{-j-1}r_2}]}e_n\\
\text{ by \eqref{sigma.dens.bd} }&\leq &\sum_{j=0}^{\infty}(Cr_2^{-1}2^{j+1})^{2(q-1)}\cdot C(2^{-j}r_2)^{\sigma}\\
&=&C'\sum^\infty_{j=0}(2^j)^{2(q-1)-\sigma}.
\end{eqnarray*}
Since $q<1+\frac{\sigma}{2}$, the geometric series converges, and we see that
$$\|e_n\|_{L^q(B_{r_2/2}(p))}\leq C_q<\infty,$$
as claimed. It follows immediately that the limit density $e_{\infty}\in L^q(B_{r_2/2}(p))$ as well, and since we already know from Lemma \ref{convergence:lemma} that $e_n\to e_{\infty}$ in $L^q_{loc}(B_{r_2}(p)\setminus \{p\})$, we easily conclude that the desired convergence \eqref{nearbub.cvg} holds for every $1\leq q<1+\frac{\sigma}{2}$.

The argument for $L^q$ convergence \eqref{nearneck.cvg} near the south pole of $\mathbb{S}^2$ is similar, once we express our estimates in an appropriate coordinate system near $S$. Denoting by $\pi_S: \mathbb{R}^2\to \mathbb{S}^2\setminus \{N\}$ the stereographic projection based at the south pole $S\in \mathbb{S}^2$, we write
$$B_r(S):=\pi_S(B_r(0)),$$
so that, e.g., $B_1(S)$ denotes the southern hemisphere in $\mathbb{S}^2$. In this notation, note that the image under $R_n: B_{r_1}(p)\to \mathbb{S}^2$ of any annulus about $c_n$ in $M$ is given by
$$R_n(B_s(c_n)\setminus B_t(c_n))=\pi(B_{\alpha_n^{-1}s}(0)\setminus B_{\alpha_n^{-1}t}(0))=B_{t^{-1}\alpha_n}(S)\setminus B_{s^{-1}\alpha_n}(S).$$
In particular, note that
$$\Xi_n=R_n(B_{\sqrt{\alpha_n}}(c_n))=\mathbb{S}^2\setminus B_{\sqrt{\alpha_n}}(S),$$
so that 
$$\widetilde e_n=|d\widetilde\Psi_n|_g^2\boldsymbol{1}_{\mathbb{S}^2\setminus B_{\sqrt{\alpha_n}}(S)}\text{ on }B_{1/2}(S).$$

Moreover, we note that for $r>\sqrt{\alpha_n}$, the image under $R_n$ of the annulus $A_{r,n}=B_r(c_n)\setminus B_{r^{-1}\alpha_n}(c_n)$ is given by
$$R_n(A_{r,n})=B_r(S)\setminus B_{r^{-1}\alpha_n}(S),$$
so that the estimate of Lemma \ref{limit_measure:lemma} has the identical form
$$\int\limits_{B_r(S)\setminus B_{r^{-1}\alpha_n}(S)}|d\widetilde{\Psi}_n|_g^2dv_g<Cr^{\sigma}$$
in terms of the local geometry near $S\in \mathbb{S}^2$. In particular, since $\widetilde{e}_n\equiv 0$ on $B_{\sqrt{\alpha_n}}(S)$, it follows exactly as it did for $e_n$ that
\begin{equation}\label{s.sigma.dens.bd}
\int\limits_{B_r(S)}\widetilde{e}_n<C r^{\sigma}\text{ for all }r>0.
\end{equation}

Likewise, the energy estimate \eqref{small.e.ann} yields a bound of the form
$$\int\limits_{B_1(S)\setminus B_{C\alpha_n}(S)}|d\widetilde\Psi_n|_g^2dv_g<\varepsilon_0$$
in our local coordinates near $S$, and we can appeal to Lemma \ref{regularity:lemma} as we did for $e_n$ to obtain the pointwise bound
\begin{equation}\label{e.tild.linfty}
d_S(x)^2\widetilde{e}_n(x)\leq C\text{ on }B_{1/2}(S).
\end{equation}
With the density bound \eqref{s.sigma.dens.bd} and the pointwise bound \eqref{e.tild.linfty} in hand, we can now argue exactly as we did for $e_n$ to conclude that
$$\|\widetilde{e}_n\|_{L^q(B_{1/2}(S))}\leq C_q\text{ for every }q<1+\frac{\sigma}{2};$$
and again, since we know from Lemma \ref{convergence:lemma} that $\widetilde{e}_n\to \widetilde{e}_{\infty}$ in $L^q_{loc}(B_{1/2}(S)\setminus \{S\})$, the desired convergence \eqref{nearneck.cvg} follows.
\end{proof}

All that remains now is to prove Lemma \ref{limit_measure:lemma}, establishing the desired energy decay bounds on the annuli $A_{r,n}=A_r(c_n)\setminus A_{r^{-1}\alpha_n}(c_n)$.

\begin{proof}[Proof of Lemma \ref{limit_measure:lemma}]

First, we recall the suspension procedure used by Parker~\cite{Parker} on pp. 607 -- 608 (see also~\cite{Gruter, Jost, Schoen}). 
Recall that $r_1>0$ is such that $g_p$ is defined on $B_{2r_1}(p)$. For the harmonic map $\Psi_n\colon M\to\mathbb{S}^{N_n}$ the Hopf differential $\mathcal H(\Psi_n)$ is defined in local coordinates by 
$$
\mathcal H(\Psi_n) = (\partial_z\Psi_n,\partial_z\Psi_n)=|\partial_x\Psi_n|^2-|\partial_y\Psi_n|^2-2i\langle\partial_x\Psi_n,\partial_y\Psi_n\rangle.
$$
The Hopf differential is holomorphic and is equal to $0$ iff $\Psi_n$ is weakly conformal. The suspension procedure associates to each $\Psi_n$ a weakly conformal harmonic map $\Upsilon_n\colon B_{2r_1}(p)\to \mathbb{S}^{N_n}\times \mathbb{C}$, so that $\mathcal H(\Upsilon_n) \equiv 0$. Namely, we select a holomorphic function $\xi_n:B_{2r_1}(p)\to \mathbb{C}$ satisfying
$$\partial_z\xi_n = -\mathcal H(\Psi_n),$$
and set in local complex coordinates 
$$\Upsilon_n(z) := (\Psi_n(z),\bar z+\xi_n).$$
Then one has
$$
\mathcal H(\Upsilon_n) = \mathcal H(\Psi_n) + \partial_z(\bar z + \xi_n)\partial_z\overline{(\bar z + \xi_n)} = \mathcal H(\Psi_n)-\mathcal H(\Psi_n) = 0
$$

Furthermore, one has $||\mathcal H(\Psi_n)||_{L^1(B_{2r_1(p)})}<CE(\Psi_n|_{B_{2r_1}(p)})$ by the definition of $\mathcal{H}(\Psi_n)$, and since $\mathcal{H}(\Psi_n)$ is holomorphic, it follows that
$$
||\mathcal H(\Psi_n)||_{L^\infty(B_{r_1}(p))}<CE(\Psi_n|_{B_{2r_1}(p)})< C.
$$
Moreover, the differential $d(\bar{z}+\xi_n)$ satisfies
$$
\frac{1}{2}|d(\bar z+\xi_n)|^2 = \partial_z(\bar z+\xi_n)\overline{\partial_z(\bar z+\xi_n)} + \partial_{\bar z}(\bar z+\xi_n)\overline{\partial_{\bar z}(\bar z+\xi_n)} = |\mathcal H(\Psi_n)|^2 + 1,
$$
so that on $B_{r_1}(p)$ one has $0<\frac{1}{2}|d(\bar z+\xi_n)|^2<C$. In particular, it follows that $\Upsilon_n$ defines a conformal immersion into $S^{N_n}\times \mathbb{C}$.

As a result, for any domain $A\subset B_{r_1}(p)$ one has 
\begin{equation}
\label{suspension_energy:ineq}
E(\Upsilon_n|_A) = \frac{1}{2}\int\limits_A |d\Psi_n|^2 + |d(\bar z+\xi_n)|^2\leqslant E(\Psi_n|A) + C\area_{g_p}(A).
\end{equation}
The utility of the suspension trick stems from the well-known fact that the images of conformal harmonic maps from two-dimensional domains are minimal surfaces. In particular, we will make use of the following fact, which follows, e.g., from the classical isoperimetric inequalities of Hoffman and Spruck \cite{HS}.

\begin{theorem}\label{HS:thm}
Let $\Sigma$ be a minimal surface in $\mathbb{S}^N\times \mathbb{C}$. Then there exist constants $\varepsilon_1,c_0>0$ independent of $N$ such that if $\area(\Sigma)<\varepsilon_1$ then $\area(\Sigma)\leqslant c_0 Length^2(\partial\Sigma)$.
\end{theorem}

\begin{remark} To see that the constants $\varepsilon_1$ and $c_0$ in Theorem \ref{HS:thm} do not depend on the dimension $N$ of the target sphere, it is enough to note that the injectivity radius and sectional curvature of $\mathbb{S}^N$ are the same for all $N$. Alternatively, one can realize the minimal surfaces $\Sigma\subset \mathbb{S}^N\times \mathbb{C}$ as surfaces of mean curvature $H\equiv 2$ in $\mathbb{R}^{N+3}$, and appeal to the isoperimetric inequalities for surfaces of constant mean curvature in Euclidean space.
\end{remark}

Now, let $r<\min\{\frac{r_1}{2},\frac{1}{2}\}$ so that $A_{r,n} = B_r(c_n)\setminus B_{r^{-1}\alpha_n}(c_n)\subset B_{r_1}(p)$. 
Since for conformal maps, energy coincides with area, by~\eqref{suspension_energy:ineq} one has that
\begin{equation*}
\begin{split}
&\area(\Upsilon_n(A_{r,n}))\leqslant E(\Psi_n|A_{r,n}) + C \area_{g_p}(A_{r,n})\\
&\leqslant \int\limits_{A_{r,n}\setminus B_{\varepsilon_n(c_n)}}\frac{1}{2}|d\Psi_n|^2_g\,dv_g + \int\limits_{B_{\varepsilon_n(c_n)}\setminus B_{2\alpha_n}(c_n)}\frac{1}{2}|d\Psi_n|^2_g\,dv_g + C\pi r^2\\
&\leqslant \int\limits_{A_{r,n}\setminus B_{\varepsilon_n(c_n)}} \frac{1}{2}\left(\left||d\Psi_n|^2_g-e_\infty\right| + e_\infty\right)\,dv_g + C_R + C\pi r^2\\
&\leqslant \varepsilon_n^2 + \int_{B_{2r}(p)}e_\infty\,dv_g + C_R + C\pi r^2.
\end{split}
\end{equation*}
 In the third line we used the definition of $\alpha_n$ and the fact that $B_{\varepsilon_n(c_n)}\setminus B_{2\alpha_n}(c_n)\subset B_{2\varepsilon_n(p)}\setminus B_{\alpha_n}(c_n)$; and in the fourth line we used~\eqref{tree:eq0}.

We now choose our normalization constant $C_R<\varepsilon_0$, by requiring that $C_R<\frac{\varepsilon_1}{3}$, where $\varepsilon_1$ is the constant from Theorem \ref{HS:thm}. Then, choosing $0<r_2<\min\{\frac{r_1}{2},\frac{1}{2}\}$ such that
$$\int_{B_{2r_2}(p)}e_{\infty}dv_g+C\pi r_2^2<\frac{\varepsilon_1}{3},$$
it follows from the preceding estimates that
$$\area(\Upsilon_n(A_{r,n}))<\varepsilon_1$$
for all $r\in (0,r_2)$. In particular, we can apply Theorem \ref{HS:thm} to the images $\Upsilon_n(A_{r,n})$ for $r\in (0,r_2)$ to conclude that
\begin{equation}
\label{HS:application}
E(\Psi_n|_{A_{r,n}})\leqslant \area(\Upsilon_n(A_{r,n}))\leqslant c_0\length^2(\partial \Upsilon_n(A_{r,n})).
\end{equation}

Next, we estimate the right-hand side of~\eqref{HS:application} in terms of the energy of $\Psi_n$. Let $r<r_2$; then by the mean value inequality, there exists $t\in [r,2r]$ such that 
\begin{eqnarray*}
t\int\limits_{\partial B_t(c_n)} |d\Upsilon_n|^2\,d\mathcal{H}^1_{g_p} &\leqslant & 2r \int\limits_{\partial B_t(c_n)}|d\Upsilon_n|^2d\mathcal{H}^1_{g_p}\\
& \leqslant & 2\int_r^{2r}\left(\int\limits_{\partial B_t(c_n)}|d\Upsilon_n|^2d\mathcal{H}_{g_p}^1\right)dt\\
&=&\int\limits_{B_{2r}(c_n)\setminus B_r(c_n)} |d \Upsilon_n|^2\,dv_{g_p}.
\end{eqnarray*}

In particular, in polar coordinates around $c_n$ one has
\begin{equation*}
\begin{split}
&\length^2(\Upsilon_n(\partial B_t(c_n)))\leqslant 2\pi \int\limits_0^{2\pi} |\partial_\theta \Upsilon_n(t,\theta)|^2d\theta = 2\pi t\int\limits_0^{2\pi} |\frac{1}{t}\partial_\theta\Upsilon_n(t,\theta)|^2 td\theta \\
& \leqslant 2\pi t\int\limits_{\partial B_t(c_n)} |d \Upsilon_n|^2\,d\mathcal{H}^1_{g_p}\leqslant 16\pi\int\limits_{B_{2r}(c_n)\setminus B_r(c_n)} |d \Upsilon_n|^2\,dv_{g_p} \\
&\leqslant 16\pi\int\limits_{B_{2r}(c_n)\setminus B_r(c_n)} |d \Psi_n|^2 + Cr^2,
\end{split}
\end{equation*}
where we used~\eqref{suspension_energy:ineq} in the last step. 

Similarly, there exists $r\leqslant t_2\leqslant 2r$ such that 
$$
\length^2(\Upsilon_n(\partial B_{t_2^{-1}\alpha_n}(c_n)))\leqslant 16\pi\int\limits_{B_{r^{-1}\alpha_n}(c_n)\setminus B_{(2r)^{-1}\alpha_n}(c_n)} |\nabla \Psi_n|^2 + C(r^{-1}\alpha_n)^2.
$$
Note that $r^{-1}\alpha_n\leqslant r$, since we are only considering $r>\sqrt{\alpha_n}$ (so that $A_{r,n}\neq\varnothing$). Thus, the previous inequality becomes
$$
\length^2(\Upsilon_n(\partial B_{t_2^{-1}\alpha_n}(c_n)))\leqslant 16\pi\int\limits_{B_{r^{-1}\alpha_n}(c_n)\setminus B_{(2r)^{-1}\alpha_n}(c_n)} |\nabla \Psi_n|^2 + Cr^2.
$$

As a result, by~\eqref{HS:application} one has
$$
E(\Psi_n |_{A_{r,n}})\leqslant E(\Psi_n|_{B_t(c_n)\setminus B_{t_2^{-1}\alpha_n}(c_n)}) \leqslant  CE(\Psi_n|_{A_{2r,n}\setminus A_{r,n}}) + Cr^2.
$$
Adding $CE(\Psi_n |_{A_{r,n}})$ to both sides of the inequality we obtain that for all $r<r_2$ one has
\begin{equation}
\label{doubling:ineq}
E(\Psi_n |_{A_{r,n}})\leqslant \theta \left(E(\Psi_n |_{A_{2r,n}}) + r^2\right),
\end{equation}
where 
$$\theta := \frac{C}{C+1}<1.$$
Since we can increase $\theta$ if necessary, we assume without loss of generality that $\theta >\frac{3}{4}$.

Let $k$ be the number so that $2^kr<r_2\leqslant 2^{k+1}r$. Applying inequality~\eqref{doubling:ineq} $k+1$ times one obtains
\begin{equation*}
\begin{split}
E(\Psi_n |_{A_{r,n}})\leqslant \theta^{k+1}E(\Psi_n |_{A_{2^{k+1}r,n}}) + \sum _{i=1}^{k+1}\left(\frac{r^2}{4}(4\theta)^i 
\right)
\end{split}
\end{equation*}

At the same time, $\log_2(r_2/r)\leqslant k+1$ and since $\theta<1$, one has 
$$
\theta^{k+1}\leqslant \theta^{\log_2(r_2/r)} = (r/r_2)^\frac{|\ln\theta|}{\ln 2}
$$
and additionally
$$
\sum_{i=1}^{k+1}(4\theta)^i =  4\theta\frac{(4\theta)^{k+1}-1}{4\theta-1}\leqslant (4\theta)^{k+1} \frac{4\theta}{4\theta-1}\leqslant 2 (4\theta)^{k+1},
$$
where we used $\theta>\frac{3}{4}$.

As a result,
\begin{equation*}
\begin{split}
E(\Psi_n |_{A_{r,n}})&\leqslant \left(\frac{r}{r_2}\right)^{\frac{|\ln\theta|}{\ln 2}}E(\Psi_n |_{A_{2r_2,n}}) + \frac{\theta^{k+1}}{2}(2^{k+1}r)^2 
\\
&\leqslant r^\sigma \left( \frac{E(\Psi_n)}{r_2^\sigma} + 2r_2^{2-\sigma}\right) 
\leqslant Cr^\sigma 
,
\end{split}
\end{equation*}
where $\sigma = \frac{|\ln\theta|}{\ln 2}>0$.

\end{proof}

\begin{remark} Note that throughout this subsection, we have invoked the spherical geometry of the target manifolds $\mathbb{S}^n$ only twice: to establish uniformity of the constants in the $\epsilon$-regularity theorem (Lemma \ref{regularity:lemma}), and to obtain uniform constants in the isoperimetric inequality for minimal surfaces in $\mathbb{S}^n\times \mathbb{C}$ (Theorem \ref{HS:thm}). We therefore expect the conclusion of Lemma \ref{bub.lem} to hold in greater generality, giving a compactness result for the energy densities associated to harmonic maps from surfaces into a larger class of target manifolds of varying dimension with suitably bounded geometry.
\end{remark}

\section{The min-max construction for the second eigenvalue}
\label{second.min-max}
\subsection{Definition and main properties of the second min-max energy}
As before, let $(M,g)$ be a fixed Riemannian surface. 
For the min-max construction corresponding to the conformal maximization of the second Laplacian eigenvalue, we consider for each $n\geq 2$ a collection 
$$\Gamma_{n,2}(M)\subset C^0([\overline{B}^{n+1}]^2,W^{1,2}(M,\mathbb{R}^{n+1}))$$
of $2(n+1)$-parameter families of maps 
$$[\overline{B}^{n+1}]^2\ni (a,b)\mapsto F_{a,b}\in W^{1,2}(M,\mathbb{R}^{n+1})$$
satisfying the boundary conditions
\begin{equation}\label{2nd.bdry.1}
F_{a,b}\equiv a\text{ if }|a|=1
\end{equation}
and
\begin{equation}\label{2nd.bdry.2}
F_{a,b}=\tau_b \circ F_{\tau_b(a),-b}\text{ if }|b|=1,
\end{equation}
where $\tau_b\in O(n+1)$ denotes reflection through the hyperplane perpendicular to $b\in \mathbb{S}^n$. Note that both conditions \eqref{2nd.bdry.1} and \eqref{2nd.bdry.2} are preserved by the gradient flow of the energies $E_{\epsilon}$, since $E_{\epsilon}$ is invariant under the action of $O(n+1)$.

This definition is motivated by Nadirashvili's computation of $\Lambda_2(\mathbb{S}^2)$ in~\cite{NadirashviliS2}. The construction was later revisited in~\cite{Petrides, GNP, GL}, but always in the context of spheres or planar domains.

For $n\geq 2$ and $\epsilon>0$, setting $\Omega_{n,2}:=(\overline{B}^{n+1})^2$, we then define the second min-max energy
\begin{equation}
\mathcal{E}_{n,2,\epsilon}(M,g):=\inf_{F\in \Gamma_{n,2}(M)}\max_{(a,b)\in\Omega_{n,2}}E_{\epsilon}(F_{a,b}),
\end{equation}
and the limit
\begin{equation}
\mathcal{E}_{n,2}(M):=\sup_{\epsilon>0}\mathcal{E}_{n,2,\epsilon}(M)=\lim_{\epsilon\to 0}\mathcal{E}_{n,2,\epsilon}(M).
\end{equation}
As with the first min-max energy $\mathcal{E}_n(M,g)$ defined in Section \ref{first.min-max}, it is easy to see that the second limiting min-max energy 
$$\mathcal{E}_{n,2}(M,g)=\mathcal{E}_{n,2}(M,[g])$$
is a conformal invariant. In what follows, we will show that it gives an upper bound for (half of) the conformal supremum $\Lambda_2(M,[g])$ of the second Laplacian eigenvalue
\begin{equation}
\Lambda_2(M,[g]):=\sup_{g\in [g]}\bar\lambda_2(M,g).
\end{equation}

\begin{proposition}\label{lambda2.below} If $\area(M,g)=1$, then 
$$2\mathcal{E}_{n,2,\epsilon}(M,g)\geq (1-2\epsilon\mathcal{E}_{n,2,\epsilon}^{1/2})\lambda_2(M,g),$$
and in particular
\begin{equation}
2\mathcal{E}_{n,2}(M,g)\geq \Lambda_2(M,[g]).
\end{equation}

\end{proposition}

The proof follows much the same lines as that of Proposition \ref{lambda.below}, with the aid of the following topological lemma.

\begin{lemma}\label{degree.lem.2} Let 
$$\Phi: \mathbb{S}^{2n+1}\to \mathbb{S}^{2n+1}$$
be a self-map of $\mathbb{S}^{2n+1}$ satisfying
\begin{equation}\label{phi.bd.1}
\Phi(a,b)=(\frac{a}{|a|},0)\text{ when }|a|\geq|b|
\end{equation}
and
\begin{equation}\label{phi.bd.2}
\Phi(\tau_b(a),-b)=(\tau_b\times \tau_b)(\Phi(a,b))\text{ when }|b|>0,
\end{equation}
where $\tau_b\in O(n+1)$ denotes reflection through the hyperplane orthogonal to $b$.
Then $\Phi$ has nonzero degree $\deg(\Phi)\neq 0$.
\end{lemma}
\begin{proof} The idea is to show that $\Phi$ has odd degree, by applying the Lefschetz fixed-point theorem to the map $\overline{\Phi}=-\Phi$. Thus, for a suitable perturbation of $\Phi$ preserving the relevant symmetries, we are interested in understanding the structure of the set
$$\mathcal{F}_-(\Phi):=\{(a,b)\in \mathbb{S}^{2n+1}\mid \Phi(a,b)=-(a,b)\}$$
comprising the fixed points of $\overline{\Phi}=-\Phi$. Our arguments are closely modeled on the proof of Claim 3 in \cite{Petrides}, suitably modified to fit our situation.

As a first step, we claim that we may take $\Phi$ to be smooth without loss of generality. To begin, note that we may deform $\Phi$ via a simple mollification procedure to a smooth map $\Phi_1\in C^{\infty}(\mathbb{S}^{2n+1},\mathbb{S}^{2n+1})$ such that
$$\Phi_1(a,b)=\Phi(a,b)=\frac{(a,0)}{|a|}\text{ on }\{|a|\geq \frac{\sqrt{3}}{2}\}\Subset \{|a|\geq |b|\}$$
and
\begin{equation}\label{phi1.c0}
\|\Phi_1-\Phi\|_{C^0}<\delta
\end{equation}
for $\delta>0$ arbitrarily small. Then, note that $\Phi_1$ automatically satisfies the symmetry \eqref{phi.bd.2} on $\{1>|a|\geq \frac{\sqrt{3}}{2}\}$, while in general for $|b|>0$, it follows from \eqref{phi.bd.2} and \eqref{phi1.c0} that
$$|\Phi_1(a,b)-(\tau_b\times\tau_b)(\Phi_1(a,b))|\leq 2\delta\text{ whenever }|b|>0.$$
In particular, provided $\delta<1$, we obtain a well-defined smooth map $\Phi_2\in C^{\infty}(\mathbb{S}^{2n+1},\mathbb{S}^{2n+1})$ by setting
$$\Phi_2(a,b):=\frac{\Phi_1(a,b)+(\tau_b\times\tau_b)(\Phi_1(a,b))}{|\Phi_1(a,b)+(\tau_b\times \tau_b)(\Phi_1(a,b))|}\text{ for }|b|>0$$
and $\Phi_2(a,0)=\Phi_1(a,0)=\frac{(a,0)}{|a|}$. The map $\Phi_2$ then satisfies 
$$\Phi_2(a,b)\equiv \frac{(a,0)}{|a|}\text{ for }|a|\geq \frac{\sqrt{3}}{2}$$
as well as the symmetry \eqref{phi.bd.2}. Moreover, by choosing $\delta>0$ sufficiently small in \eqref{phi1.c0}, it is clear that $\Phi_2$ must be $C^0$-close--and in particular, homotopic--to the original map $\Phi$.

So, suppose now that $\Phi\in C^{\infty}(\mathbb{S}^{2n+1},\mathbb{S}^{2n+1})$ is a smooth map satisfying the symmetry \eqref{phi.bd.2} as well as
\begin{equation}\label{phi.bd.1.2}
\Phi(a,b)\equiv \frac{(a,0)}{|a|}\text{ for }|a|\geq \frac{\sqrt{3}}{2}.
\end{equation}
We wish to deform $\Phi$ to a map $\Psi\in C^{\infty}(\mathbb{S}^{2n+1},\mathbb{S}^{2n+1})$ satisfying a transversality condition on the set
\begin{equation}\label{trans.map}
\mathcal{F}_-(\Psi):=\{(a,b)\in \mathbb{S}^{2n+1}\mid \Psi(a,b)=-(a,b)\},
\end{equation}
while continuing to satisfy the symmetry \eqref{phi.bd.2} near $\mathcal{F}_-(\Psi)$. Namely, to apply the Lefschetz fixed point theorem to the map $\overline{\Psi}=-\Psi$, we need to ensure the nondegeneracy of the linear map
$$(d\Psi_x+I): \mathcal{H}_x\to \mathcal{H}_x$$
at each point $x\in \mathcal{F}_-(\Psi)$, where we denote by $\mathcal{H}_x=T_x\mathbb{S}^{2n+1}$ the hyperplane in $\mathbb{R}^{2n+2}$ perpendicular to $x$. To this end, following \cite{Petrides}, we write our map $\Phi: \mathbb{S}^{2n+1}\to \mathbb{S}^{2n+1}$ as
$$\Phi(x)=X_{\Phi}(x)+\eta_{\Phi}(x) x,$$
where $\langle X_{\Phi}(x),x\rangle=0$. Note that a point $x\in\mathcal{F}_-(\Phi)$ is characterized by the conditions $X_{\Phi}(x)=0$ and $\eta_{\Phi}(x)=-1$, and it's easy to see that
$$d\Phi_x+I=d(X_{\Phi})_x\text{ for }x\in \mathcal{F}_-(\Phi).$$

Now, define for each $k=0,\ldots,2n+1$ and $\alpha>0$ the set
$$\mathcal{C}_k^{\alpha}:=\{(a,b)\in \mathbb{S}^{2n+1}\mid |b|\geq \alpha,\text{ }\langle b, e_k\rangle \geq \alpha|b|\},$$
where we denote by $e_0,\ldots,e_n$ the standard unit vectors in $\mathbb{R}^{n+1}$. By choosing $\alpha=\alpha(n)\in (0,1/8)$ sufficiently small, we may arrange that
$$\bigcup_{j=0}^{2n+1}\mathcal{C}_j^{2\alpha}\cup(-\mathcal{C}_j^{2\alpha})=\{(a,b)\in \mathbb{S}^{2n+1}\mid |b|\geq 2\alpha\}.$$
Starting from the vector field $X_0:=X_{\Phi}$ on $\mathbb{S}^{2n+1}$, note (appealing, as in \cite{Petrides}, to Sard's theorem in appropriate coordinate charts) that one may easily deform $X_0$ via a perturbation supported in $\mathcal{C}_0^{\alpha}$ to a smooth vector field $X_1$ which is transverse to the zero section in $\mathcal{C}_0^{2\alpha}\Subset \mathcal{C}_0^{\alpha}$. We may then define $X_1$ on $-\mathcal{C}_0^{\alpha}$ by setting
$$X_1(a,b)=(\tau_b\times \tau_b)(X_1(\tau_b(a),-b))\text{ for }(a,b)\in \mathcal{C}_0^{\alpha},$$
and set $X_1=X_0\text{ on }\mathbb{S}^{2n+1}\setminus [\mathcal{C}_0^{\alpha}\cup -\mathcal{C}_0^{\alpha}]$, to obtain a new, smooth tangent vector field on $\mathbb{S}^{2n+1}$ such that 
$$X_1=X_0\text{ on }\mathbb{S}^{2n+1}\setminus [\pm \mathcal{C}_0^{\alpha}],$$
$$X_1\text{ is transverse to the zero section in }\pm \mathcal{C}_0^{2\alpha},$$
and
$$X_1(\tau_b(a),-b)=(\tau_b\times \tau_b)(X_1(a,b))\text{ when }|b|>0,$$
and we may also ask that $X_1$ remains arbitrarily close to $X_0=X_{\Phi}$ in $C^1$.
Repeating the process, we obtain inductively a sequence of tangent vector fields $X_1,\ldots, X_{2n+2}$ such that
$$X_{j+1}=X_j\text{ on }\mathbb{S}^{2n+1}\setminus [\pm \mathcal{C}_j^{\alpha}],$$
$$X_{j+1}\text{ is transverse to the zero section on }\bigcup_{k=0}^j[\pm \mathcal{C}_k^{2\alpha}],$$
$$X_{j+1}(\tau_b(a),-b)=(\tau_b\times \tau_b)(X_{j+1}(a,b))\text{ when }|b|>0,$$
and
$$\|X_{j+1}-X_j\|_{C^1}<\epsilon_j$$
for some $\epsilon_j>0$ which we can take arbitrarily small. Finally, provided each $\epsilon_j>0$ is taken sufficiently small, we define a map $\Psi\in C^{\infty}(\mathbb{S}^{2n+1},\mathbb{S}^{2n+1})$ by
$$\Psi(a,b):=\frac{X_{2n+2}(a,b)+\eta(a,b)(a,b)}{|X_{2n+2}(a,b)+\eta(a,b)|},$$
and readily check that this map satisfies the symmetry
\begin{equation}\label{same.sym}
\Psi(\tau_b(a),-b)=(\tau_b\times\tau_b)(\Psi(a,b))\text{ for $|a|<1$}
\end{equation}
as well as the transversality condition
$$(d\Psi_x+I): \mathcal{H}_x\to \mathcal{H}_x\text{ is invertible for every }x\in \mathcal{F}_-(\Psi),$$
where we use that $\mathcal{F}_-(\Psi)\cap \{|a|=1\} = \varnothing$.
Therefore, $\overline{\Psi}=-\Psi$ satisfies the desired transversality condition at its fixed point set.
Since $\Psi$ is $C^1$-close to the map $\Phi$, it's evidently homotopic to $\Phi$, so once we've shown that $\Psi$ is homotopically nontrivial, we'll complete the proof of the lemma.

Finally, it is an easy consequence of the Lefschetz-Hopf fixed point theorem that for a self-map $\overline{\Psi}$ of the sphere $\mathbb{S}^{2n+1}$ satisfying the natural transversality condition at its fixed point set $Fix(\overline{\Psi})=\{x\in \mathbb{S}^{2n+1}\mid \overline{\Psi}(x)=x\}$, the degree $\deg(\overline{\Psi})$ satisfies
\begin{equation}\label{lefschetz.cor}
\deg(\overline{\Psi})\equiv \# Fix(\overline{\Psi})+1\text{ }\mod 2.
\end{equation}
Now, taking $\overline{\Psi}=-\Psi$ for the map $\Psi: \mathbb{S}^{2n+1}\to \mathbb{S}^{2n+1}$ obtained above, it follows from the symmetry \eqref{same.sym} that $(a,b)\in Fix(\overline{\Psi})$ if and only $(\tau_b\times \tau_b)(a,b)=(\tau_b(a),-b)\in Fix(\overline{\Psi})$ as well. In particular, it follows that $\# Fix(\overline{\Psi})$ must be even, so by \eqref{lefschetz.cor}, $\overline{\Psi}$ must have odd degree. Thus, $\Psi$ and our initial map $\Phi$ must have odd degree as well, and in particular must be homotopically nontrivial.

\end{proof}

With this topological lemma in place, we turn now to the proof of Proposition \ref{lambda2.below}.

\begin{proof}[Proof of Proposition \ref{lambda2.below}]

Suppose $\area(M,g)=1$, and let 
$$[\overline{B}^{n+1}]^2\ni (a,b)\mapsto F_{a,b}\in W^{1,2}(M,\mathbb{R}^{n+1})$$
be a family in $\Gamma_{n,2}(M)$. Now, let $\phi_1$ be an eigenfunction for the Laplacian $\Delta_g$ on $(M,g)$ corresponding to the first (nonzero) eigenvalue, and consider the map $\mathcal{I}: [\overline{B}^{n+1}]^2\to \mathbb{R}^{2(n+1)}$ given by
$$\mathcal{I}(a,b):=\left(\int_M F_{a,b},\int_M \phi_1 F_{a,b}\right).$$
We claim now that $\mathcal{I}(a,b)=0$ for some $(a,b)\in [\overline{B}^{n+1}]^2$. 

Suppose, to the contrary, that $\mathcal{I}$ is nowhere vanishing; then we may define a continuous map 
$$P: [\overline{B}^{n+1}]^2\to \mathbb{S}^{2n+1};\text{ }P(a,b):=\frac{\mathcal{I}(a,b)}{|\mathcal{I}(a,b)|}.$$
In particular, restricting $P$ to the boundary of $[\overline{B}^{n+1}]^2$ and identifying the boundary with $\mathbb{S}^{2n+1}$ in the obvious way, we see that the resulting map 
\begin{equation}\label{phi.map.def}
\Phi: \mathbb{S}^{2n+1}\to \mathbb{S}^{2n+1}\text{; }\Phi(a,b):=P\left(\frac{(a,b)}{\max\{|a|,|b|\}}\right)
\end{equation}
must be homotopically trivial. 

On the other hand, since $F\in \Gamma_{n,2}(M)$, it follows from \eqref{2nd.bdry.1} and \eqref{2nd.bdry.2} and the definition of $\mathcal{I}$ that 
$$P(a,b)=(a,0)\text{ for }(a,b)\in \mathbb{S}^n\times \overline{B}^{n+1}$$
and
$$P(\tau_b(a),-b)=(\tau_b\times \tau_b)(P(a,b)),$$
since the averaging maps $W^{1,2}(M,\mathbb{R}^{n+1})\to \mathbb{R}^{n+1}$ given by $u\mapsto \int_Mu$ and $u\mapsto \int_M\phi_1 u$ commute with linear transformations of $\mathbb{R}^{n+1}$. In particular, we easily deduce that the map $\Phi: \mathbb{S}^{2n+1}\to \mathbb{S}^{2n+1}$ given by \eqref{phi.map.def} satisfies the hypotheses of Lemma \ref{degree.lem.2}, and therefore must be homotopically \emph{nontrivial}, by the lemma. Thus, we see that the map $\mathcal{I}:[\overline{B}^{n+1}]^2\to \mathbb{R}^{2(n+1)}$ must have a zero somewhere.

We've now shown that for any family $F\in \Gamma_{n,2}(M)$, there exists some $(a,b)\in \overline{B}^{2(n+1)}$ for which the map $u=F_{a,b}\in W^{1,2}(M,\mathbb{R}^{n+1})$ satisfies
$$\int_Mu=\int_M\phi_1u=0\in \mathbb{R}^{n+1}.$$
That is, each scalar component $u^j$ of $u=F_{a,b}$ is $L^2$ orthogonal to $1$ and $\phi_1$, from which it follows that
$$\int_M |du^j|^2\geq \lambda_2(M,g)\int_M (u^j)^2\text{ for each }j=0,\ldots,n,$$
where $\lambda_2(M,g)$ denotes the second nontrivial eigenvalue of the Laplacian $\Delta_g$. In particular, summing from $j=0$ to $n$, we have the lower bound
\begin{equation}\label{ener.eigen.bd}
\int_M |du|^2\geq \lambda_2(M,g)\int_M |u|^2
\end{equation}
for the full energy of the map $u=F_{a,b}$. On the other hand, recalling the definition of the functionals $E_{\epsilon}$, we see that
\begin{eqnarray*}
\int_M|u|^2&\geq &1-\int_M|1-|u|^2|^2\\
&\geq &1-2\epsilon E_{\epsilon}(u)^{1/2}\\
&\geq &1-2\epsilon [\max_{(a,b)}E_{\epsilon}(F_{a,b})]^{1/2},
\end{eqnarray*}
and combining this with the preceding estimate gives
\begin{eqnarray*}
2\max_{a,b}E_{\epsilon}(F_{a,b})&\geq &\int_M |du|^2\\
&\geq &\lambda_2(M,g)(1-2\epsilon [\max_{a,b}E_{\epsilon}(F_{a,b})]^{1/2}.
\end{eqnarray*}

Applying the preceding inequality to a sequence of families $F^j\in \Gamma_{n,2}(M)$ with 
$$\max_{(a,b)}E_{\epsilon}(F^j_{a,b})\to \mathcal{E}_{n,2,\epsilon}(M,g),$$
we obtain the desired estimate
$$2\mathcal{E}_{n,2,\epsilon}(M,g)\geq(1-2\epsilon \mathcal{E}_{n,2,\epsilon}(M,g)^{1/2})\lambda_2(M,g).$$
Moreover, recalling that 
$$\mathcal{E}_{n,2}(M,g)=\sup_{\epsilon>0}\mathcal{E}_{n,2,\epsilon}(M,g)$$
is a conformal invariant, taking the limit as $\epsilon\to 0$ yields the bound
\begin{equation}
2\mathcal{E}_{n,2}(M,g)\geq \Lambda_2(M,[g]),
\end{equation}
completing the proof of the proposition.

\end{proof}

Next, we use a variant of a construction of Nadirashvili~\cite{NadirashviliS2} to provide uniform upper bounds for the min-max energies $\mathcal{E}_{n,2,\epsilon}(M,g)$ as $\epsilon\to 0$, giving the finiteness of the limiting min-max energies $\mathcal{E}_{n,2}(M)$.

\begin{proposition}\label{e2.upper}
For any conformal class $[g]$ on $M$, we have the upper bound
$$\mathcal{E}_{n,2}(M,[g]):=\sup_{\epsilon>0}\mathcal{E}_{n,2,\epsilon}(M,g)\leq 2V_c(n,M,[g]).$$
\end{proposition}

\begin{proof} Similar to the proof of Proposition \ref{vc.above}, we'll first construct a weakly continuous family of conformal maps from $M$ to $\mathbb{S}^n$ satisfying the requisite symmetry assumptions and desired energy bounds, then produce strongly continuous approximations $M\to \mathbb{R}^{n+1}$ via mollification.

To this end, consider the family of maps $B^{n+1}\ni a \mapsto T_a\in Lip(\mathbb{S}^n,\mathbb{S}^n)$ defined as follows: 
$$T_a=\mathrm{Id}\text{ on the spherical cap }C_a:=\{x\in \mathbb{S}^n\mid \langle x,a\rangle \leq |a|-|a|^2\}$$
and on the complementary cap $\mathbb{S}^n\setminus C_a$, $T_a$ is the unique conformal reflection
$$T_a=Rf_a: \mathbb{S}^n\setminus C_a\to C_a$$
which acts as the identity on the boundary $\partial C_a$. Note that
$$T_0\equiv\mathrm{ Id} :\mathbb{S}^n\to \mathbb{S}^n,$$
since $C_0=\mathbb{S}^n$ is the whole sphere, and when $|a|=1$, note that $C_a$ defines the hemisphere opposite $a\in \mathbb{S}^n$, and $T_a$ acts on $\mathbb{S}^n\setminus C_a$ by linear reflection $\tau_a$ through the hyperplane perpendicular to $a$. 

Now, as in Proposition \ref{vc.above}, we consider the family of conformal maps
$$B^{n+1}\ni \xi \mapsto G_{\xi}(x):=\frac{(1-|\xi|^2)}{|x+\xi|^2}(x+\xi)+\xi,$$
and we define a new family
$$[B^{n+1}]^2\ni (a,b)\mapsto \Upsilon_{a,b}\in Lip(\mathbb{S}^n,\mathbb{S}^n)$$
by the composition
\begin{equation}
\Upsilon_{a,b}:=G_a\circ T_b
\end{equation}
of the two $(n+1)$-parameter families.

For $n\geq 2$, fix now a branched conformal immersion
$$\phi: M\to \mathbb{S}^n$$
from our Riemann surface $(M^2,g)$ into $\mathbb{S}^n$, and consider the family of maps
\begin{equation}\label{ups.def}
[B^{n+1}]^2\ni (a,b)\mapsto F_{a,b}:= \Upsilon_{a,b}\circ \phi: M\to \mathbb{S}^n.
\end{equation}
Though the family $(a,b)\mapsto F_{a,b}$ will not define a strongly continuous family in $W^{1,2}(M,\mathbb{S}^n)$ (indeed, we expect $F_{a,b}$ to exhibit some energy concentration both as $|a|\to 1$ and as $b\to 0$), it is not difficult to see that the energy $E(F_{a,b})$ can be bounded above in terms of the conformal volume $V_c(n,\phi)$. Indeed, it follows from the definition of the maps $\Upsilon_{a,b}$ that
\begin{eqnarray*}
E(F_{a,b})&=&\int_{\phi^{-1}(C_b)}\frac{1}{2}|d(G_a\circ \phi)|^2+\int_{\phi^{-1}(\mathbb{S}^n\setminus C_b)}\frac{1}{2}|d(G_a\circ Rf_b\circ \phi)|^2\\
&\leq & E(G_a\circ \phi)+E(G_a\circ Rf_b\circ \phi)\\
&\leq & V_c(n,\phi)+V_c(n,\phi),
\end{eqnarray*}
since $G_a$ and $G_a\circ Rf_b$ are both conformal automorphisms of $\mathbb{S}^n$ (unless $|a|=1$, in which case $F_{a,b}\equiv a$ is constant). In particular, it follows that
\begin{equation}\label{f.bd.1}
\sup_{(a,b)\in [B^{n+1}]^2}E(F_{a,b})\leq 2V_c(n,\phi).
\end{equation}
Moreover, note that $G_a\equiv a$ when $|a|=1$, and when $|b|=1$, it follows from the definition of $G_a$ and $T_b$ that
$$G_{\tau_b(a)}\circ T_{-b}=G_{\tau_b(a)}\circ \tau_b\circ T_b=\tau_b\circ G_a\circ T_b$$
for any $a\in B^{n+1}$. Hence, by definition of $F_{a,b}$, we have
\begin{equation}\label{pre.fam.sym}
F_{a,b}\equiv a\text{ if }|a|=1\text{ and }F_{a,b}=\tau_b\circ F_{\tau_b(a),-b}\text{ if }|b|=1.
\end{equation}

To produce families in $\Gamma_{n,2}(M)$ satisfying the desired energy bounds, we will once again mollify the weakly continuous family $F_{a,b}=\Upsilon_{a,b}\circ \phi$ to obtain strongly continuous families in $W^{1,2}(M,\mathbb{R}^{n+1})$ satisfying the same symmetries and energy bound. Namely, let $K_t(x,y)$ again denote the heat kernel on $(M,g)$, and denote by
$$\Phi^t: L^1(M,\mathbb{R}^{n+1})\to C^2(M,\mathbb{R}^{n+1})$$
the mollification map
$$(\Phi^tF)(x):=\int_MF(y)K_t(x,y)dy$$
for $t>0$. Then, letting $F\in C^0(\overline{B}^{2(n+1)}, L^1(M, \mathbb{S}^n))$ be a family of the form $F_{a,b}:=\Upsilon_{a,b}\circ \phi$ for some branched conformal immersion $\phi: M\to \mathbb{S}^n$, it's easy to see--as in the proof of Proposition \ref{vc.above}--that the mollified families
$$F^t_{a,b}:=\Phi^tF_{a,b}:=\Phi^t(\Upsilon_{a,b}\circ \phi)$$
define strongly continuous assignments $\overline{B}^{2(n+1)}\to W^{1,2}(M,\mathbb{R}^{n+1})$, and inherit from $F_{a,b}$ the symmetries \eqref{pre.fam.sym}. 

In particular, it follows that $F^t\in \Gamma_{n,2}(M)$ for each $t>0$. Since $F^t_{a,b}$ is obtained from the heat flow with initial data $F_{a,b}$, we also have the energy bound
$$\int_M \frac{1}{2}|dF^t_{a,b}|^2\leq \int_M \frac{1}{2}|dF_{a,b}|^2$$
while arguments identical to those in the proof of Proposition \ref{vc.above} show that
$$\lim_{t\to 0}\max_{(a,b)\in \overline{B}^{2(n+1)}}\int_M (1-|F^t_{a,b}|^2)^2=0.$$
Recalling that the initial family $F_{a,b}$ satisfies the energy bound \eqref{f.bd.1}, we deduce that, for any $\epsilon>0$,
\begin{eqnarray*}
\mathcal{E}_{n,2,\epsilon}(M,g)&\leq & \inf_{t>0}\max_{a,b}E_{\epsilon}(F^t_{a,b})\\
&\leq &\max_{a,b}\int_M\frac{1}{2}|dF_{a,b}|^2\\
&\leq &2V_c(n,\phi).
\end{eqnarray*}
Since the bound holds for arbitrary $\epsilon>0$, it follows that
$$\mathcal{E}_{n,2}(M,[g])=\sup_{\epsilon>0}\mathcal{E}_{n,2,\epsilon}(M,g)\leq 2V_c(n,\phi),$$
and taking the infimum over all branched conformal immersions $\phi: M\to \mathbb{S}^n$ gives the desired estimate
$$\mathcal{E}_{n,2}(M,[g])\leq 2V_c(n,M,[g]).$$

\end{proof}

\subsection{Existence of min-max harmonic maps}
\label{second.min-max:properties}
We have already seen in Proposition \ref{ps.fred.prop} that the functionals $E_{\epsilon}$ are $C^2$ functionals on the Hilbert space $W^{1,2}(M,\mathbb{R}^{n+1})$, satisfying the technical conditions needed to produce critical points with index bounds via classical min-max techniques. Moreover, since the functionals $E_{\epsilon}$ are invariant under the action of $O(n+1)$ on $\mathbb{R}^{n+1}$, we see that the collection of $(2n+2)$-parameter families $\Gamma_{n,2}(M)$ is preserved under the gradient flow of $E_{\epsilon}$, so we can again appeal to standard results in critical point theory (again, see \cite{Ghou}, Chapter 10) to arrive at the following existence result.

\begin{proposition}\label{2nd.ex.thm} For each $\epsilon>0$, there exists a nontrivial critical point $\Psi_{\epsilon}: M^2\to \mathbb{R}^{n+1}$ of $E_{\epsilon}$ on $(M,g)$, of energy
\begin{equation}
E_{\epsilon}(\Psi_{\epsilon})=\mathcal{E}_{n,2,\epsilon}(M,g)
\end{equation}
and Morse index
\begin{equation}
\ind_{E_{\epsilon}}(\Psi_{\epsilon})\leq 2n+2.
\end{equation}
\end{proposition}

Finally, combining this basic existence result with Propositions \ref{lambda2.below} and \ref{e2.upper}, and invoking the bubbling analysis of Lemma \ref{bubble.business}, we take the limit of these maps as $\epsilon\to 0$, arriving at the following conclusion.

\begin{theorem}
\label{En2:thm}
 For any closed Riemannian surface $(M,[g])$ of positive genus and any $n\geqslant 2$, there exists a harmonic map $\Psi_n\colon M\to \mathbb{S}^n$ and harmonic maps $\phi_1,\ldots, \phi_k: M\to \mathbb{S}^n$ such that
$$
\frac{1}{2}\Lambda_2(M,[g])\leqslant \mathcal{E}_{n,2}(M,[g])=E(\Psi_n)+\sum_{j=1}^kE(\phi_j)\leqslant 2V_c(n,M,[g])
$$
and
$$\ind_E(\Psi)+\sum_{j=1}^k\ind_E(\phi_j)\leqslant 2n+2.$$
Moreover, if $n\geq 9$, then we have $k=0$ or $1$, and if $k=1$, then $\phi_1\colon \mathbb{S}^2\to \mathbb{S}^n$ is a totally geodesic embedding.

\end{theorem}
\begin{proof}
The proof is similar to the proof of Theorem~\ref{eigenvalue:thm}. The first part easily follows from Proposition~\ref{2nd.ex.thm} and Lemma~\ref{bubble.business}.
Assume $n\geqslant 9$.
Since $2(n+1)<3(n-2)$ for $n\geqslant 9$, Propositions~\ref{indEM:prop} and~\ref{indES2:prop} imply that one of the following three possibilities must hold: either $\Psi_n$ is constant, $k=2$, and $\phi_1,\phi_2$ are equatorial bubbles; or $k=1$, $\phi_1$ is an equatorial bubble; or $k=0$.

The first case, in which the energy $\mathcal{E}_{n,2}$ is achieved by two equatorial bubbles, can be ruled out using Theorem~\ref{rigidity:thm}. Indeed by~\eqref{Lambdak:condition} applied to $\Lambda_2(M,[g])$ one has $\Lambda_2(M,c)>16\pi$; thus, if $\mathcal{E}_{n,2}$ is achieved by two equatorial bubbles, then one has
$$
8\pi<\frac{1}{2}\Lambda_2(M,[g])\leqslant\mathcal E_{n,2} = E(\phi_1) + E(\phi_2) =  8\pi,
$$
which is a contradiction.
\end{proof}

\subsection{Stabilization for $\mathcal E_{n,2}$} 
\label{second.min-max:stabilization}
Similarly to Section~\ref{stabilization:sec} we will conclude that the inequality
\begin{equation}
\label{En2:bound}
\Lambda_2(M,[g])\leqslant 2\mathcal E_{n,2}(M,g)
\end{equation}
is an equality for large $n$.

The proof of the following proposition is identical to Proposition~\ref{monotoneEn:prop} 
\begin{proposition}
The quantity $\mathcal E_{n,2}(M,[g])$ is non-increasing in $n$.
\end{proposition}

Next, note that one of the two cases in the conclusion of Theorem~\ref{En2:thm} must hold for infinitely many $n\in \mathbb{N}$. Thus, for any $(M,[g])$, we know that at least one of the following must hold:
\begin{itemize}
\item[{\bf Case 1:}] There exists a sequence $n_k\to\infty$ such that $\mathcal E_{n_k,2}(M,[g]) = E(\Psi_{n_k}) + 4\pi$, where 
$\Psi_{n_k}\colon (M,[g])\to\mathbb{S}^{n_k}$ is a harmonic map with $\ind_E(\Psi_{n_k})\leqslant n_k+4$;
\item[{\bf Case 2:}] There exists a sequence $n_k\to\infty$ such that $\mathcal E_{n_k,2}(M,[g]) = E(\Psi_{n_k})$, where 
$\Psi_{n_k}\colon (M,[g])\to\mathbb{S}^{n_k}$ is a harmonic map with $\ind_E(\Psi_{n_k})\leqslant 2n_k+2$.
\end{itemize}

Assuming Case 1, the same arguments as in Section~\ref{stabilization:sec} yield the existence of $\Psi_n$ such that $\mathcal E_{n,2} = E(\Psi_n) + 4\pi$ and $\ind_S(\Psi_n) = 1$. Then one has
$$
\Lambda_2(M,[g])\leqslant 2\mathcal E_{n,2}(M,[g]) = \bar \lambda_1(M,g_{\Psi_n}) + 8\pi \leqslant \Lambda_1(M,[g])+8\pi\leqslant \Lambda_2(M,[g]).
$$
In particular, inequality~\eqref{En2:bound} is an equality.

Assuming Case 2, the arguments of Section~\ref{stabilization:sec} yield the existence of $\Psi_n$ such that $\mathcal E_{n,2} = E(\Psi_n)$ and $\ind_S(\Psi_n) \leqslant 2$. If $\ind_S(\Psi_n) = 1$, then one has
$$
\Lambda_2(M,[g])\leqslant 2\mathcal E_{n,2}(M,[g]) = \bar\lambda_1(M,g_{\Psi_n})\leqslant \Lambda_1(M,[g]),
$$
which is a contradiction. If $\ind_S(\Psi_n) = 2$, then one has
$$
\Lambda_2(M,[g])\leqslant 2\mathcal E_{n,2}(M,[g]) = \bar\lambda_2(M,g_{\Psi_n})\leqslant \Lambda_2(M,[g]).
$$
In particular, inequality~\eqref{En2:bound} is an equality. 

As a result, we obtain
\begin{theorem}
For any $(M,[g])$ there exists $N$ such that for all $n\geqslant N$ one has
$$
\frac{1}{2}\Lambda_2(M,[g]) = \mathcal E_{n,2}(M,[g]).
$$
\end{theorem}

\section{Applications}\label{app.sec}

The starting point for the geometric applications of our min-max characterization for $\Lambda_k(M,c)$ is the following theorem, showing that for $k=1,2$, the supremum $\Lambda_k(M,c)$ of the eigenvalue $\bar{\lambda}_k(M,g)$ over the conformal class $c=[g]$ is an upper bound for the generalized eigenvalues $\lambda_k(M,c,\mu)$ (recall the definition in~\eqref{MeasureRayleigh:quotient}) associated to any Radon probability measure $\mu$. 

\begin{theorem}\label{mu.eigen.bd}
\label{measures:thm}
Let $\mu$ be an admissible Radon measure of unit mass $\mu(M)=1$. Then one has
$$
\lambda_1(M,c,\mu)\leqslant \Lambda_1(M,c),
$$
with equality if only if 
$$\lambda_1(M,c,\mu) \mu=|du|_g^2\,dv_g$$
for some harmonic map $u: (M,g)\to \mathbb{S}^n$ of spectral index $1$.
Furthermore,
$$\lambda_2(M,c,\mu)\leqslant \Lambda_2(M,c),$$
with equality if and only if
$$\lambda_2(M,c,\mu)\mu=|du|_g^2\,dv_g$$
for some harmonic map $u:(M,g)\to \mathbb{S}^n$ of spectral index $2$.

\end{theorem}
We postpone the proof of the theorem to Section \ref{meas.thm.pf}. This theorem has a nice application to the study of Steklov eigenvalues, which we describe in the following section.

\subsection{Steklov eigenvalues}\label{steklov.sec} Given a sufficiently regular (e.g. Lipschitz) domain $\Omega\subset M$ (or any surface with boundary) the Steklov eigenvalues $\sigma_k(\Omega, g)$ are defined via Rayleigh quotients, as
 \begin{equation}
 \label{SteklovRayleigh:quotient}
\sigma_k(\Omega,g)= \inf_{G_{k+1}}\sup_{u\in G_{k+1}\setminus\{0\}}\frac{\displaystyle\int_{\Omega}|\nabla u|^2_g\,dv_g}{\displaystyle\int_{\partial\Omega} u^2\,ds_g},
\end{equation}
where the infimum is taken over $(k+1)$-dimensional subspaces $G_{k+1}\subset C^\infty(\Omega)$ that remain $(k+1)$-dimensional in $L^2(\partial \Omega)$. It is not difficult to check that the eigenvalues $\sigma_k(\Omega,g)$ defined by \eqref{SteklovRayleigh:quotient} correspond to the spectrum of the Dirichlet-to-Neumann map
$$C^{\infty}(\partial\Omega) \ni \varphi \mapsto \frac{\partial \hat{\varphi}}{\partial \nu}\in C^{\infty}(\partial \Omega)\text{, where }\Delta \hat{\varphi}=0\text{ in }\Omega,\text{ }\hat{\varphi}|_{\partial\Omega}=\varphi.$$
Similar to the normalized Laplacian eigenvalues, one defines the normalized Steklov eigenvalues by
$$
\bar\sigma_k(\Omega,g) = \sigma_k(\Omega,g)\length(\partial \Omega, g).
$$
The theory of optimal eigenvalue inequalities for $\bar\sigma_k$ is very much parallel to that of $\bar\lambda_k$, and has received considerable attention in recent years, in connection with the study of free boundary minimal surfaces in Euclidean balls; see~\cite{GP,FS} for some recent surveys. 

Let $\mu = \mu_{\partial\Omega}$ to be the length density $s_g$ of $\partial\Omega$. Let $\Omega\subset M$ and assume the measure $\mu_{\partial\Omega}$ is admissible, i.e. the trace map $W^{1,2}(M,g)\to L^2(\partial \Omega, g)$ is compact. This is satisfied, for example, provided $\Omega$ is Lipschitz. Comparing~\eqref{MeasureRayleigh:quotient} and~\eqref{SteklovRayleigh:quotient} one easily sees that
$$
\sigma_k(\Omega,g)\leqslant \lambda_k(M,[g],\mu_{\partial \Omega}).
$$
Since the measure $\mu_{\partial \Omega}$ can not have full support, Theorem~\ref{measures:thm} has the following corollary.
\begin{theorem}
\label{Steklov:thm}
For any $\Omega\subset M$ such that the trace map $W^{1,2}(M,g)\to L^2(\partial \Omega, g)$ is compact and for $k=1,2$ one has
$$
\bar\sigma_k(\Omega,g)< \Lambda_k(M,c)
$$ 
for every $g\in c$.
\end{theorem}
Theorem~\ref{Steklov:thm} gives a sharp bound, independent of the number of boundary components of $\partial\Omega$. See e.g.~\cite{KarSteklov, FS, Has} for other bounds on Steklov eigenvalues.

\subsection{Applications to the existence of maximal metrics for Steklov eigenvalues}
\label{Steklov_applications:sec}

As discussed in the introduction, the following result--obtained in the recent preprint~\cite{GLS}--implies that Theorem \ref{Steklov:thm} above is sharp.

\begin{theorem}[Girouard, Lagac\'e~\cite{GLS}]
\label{GLS:thm}
Given a surface $(M,g)$, there exists a sequence of smooth domains $\Omega_n\subset M$ such that for all $k$ one has
$$
\lim_{n\to\infty}\bar\sigma_k(\Omega_n,g)\to\bar\lambda_k(M,g).
$$
\end{theorem}

In this section we explore further applications of Theorems~\ref{Steklov:thm} and~\ref{GLS:thm}. Some of these statements also appear in~\cite{GLS}. Let us first introduce the notation
$$
\Sigma_k(\Omega,c) = \sup_{g\in c}\bar\sigma_k(\Omega,g).
$$
The following theorem is an analog of Theorem~\ref{existence:thm} for Steklov eigenvalues.

\begin{theorem}[Petrides~\cite{PetridesSteklov}]
Assume that 
\begin{equation}
\label{SigmaK:condition}
\Sigma_k(\Omega,c)>\Sigma_{k-1}(\Omega,c)+2\pi.
\end{equation}
Then there exists a metric $g\in c$ such that $\bar\sigma_k(\Omega,g) = \Sigma_k(\Omega,c)$.
\end{theorem}

The following proposition states that the condition~\eqref{SigmaK:condition} is satisfied for many conformal classes.

\begin{proposition}
Let $(M, c_0)$ be a surface with a fixed conformal class $c_0$. Then for any $0<k\leqslant 3$ there exists $b_0\geqslant 0$ such that for any $b\geqslant b_0$ there exists $(\Omega_b, c_b)\subset (M,c_0)$ such that  $\Omega$ has exactly $b$ boundary components and the condition~\eqref{SigmaK:condition}  is satisfied for $(\Omega_b,c_b)$.
\end{proposition}
\begin{proof}
For any $(\Omega,c)\subset (M,c_0)$ by Theorem~\ref{Steklov:thm} one has
$$
\Sigma_{k-1}(\Omega,c)<\Lambda_{k-1}(M,c)\leqslant \Lambda_k(M,c) - 8\pi,
$$
where in the last inequality we used~\eqref{Lambdak:condition}. Let $g\in c$ be a metric such that $\bar\lambda_k(M,g)+2\pi>\Lambda_k(M,c)$. By Theorem~\ref{GLS:thm} there exists $\Omega_0\subset M$ such that 
$$
\bar\sigma_k(\Omega_0, g)\geqslant \bar\lambda_k(M,g)-2\pi>\Lambda_k(M,c)-4\pi.
$$
Combining the two previous inequalities one has
$$
\Sigma_k(\Omega_0,c)\geqslant \bar\sigma_k(\Omega_0, g)>\Sigma_{k-1}(\Omega_0,c)+4\pi.
$$
Set $b_0$ to be the number of boundary components of $\Omega_0$. If $b\geqslant b_0$, then by the results of~\cite{BGT} (which continue to hold in manifold setting, see the proof of~\cite[Lemma 3.1]{GLS})
one can cut out several holes in $\Omega_0$ to obtain $\Omega_b$ such that
$$
\bar\sigma_k(\Omega_b,g)\geqslant \bar\sigma_k(\Omega_0, g)-2\pi.
$$
For such $\Omega_b$ one has
$$
\Sigma_k(\Omega_b,c)\geqslant \bar\sigma_k(\Omega_b,g)>\Sigma_{k-1}(\Omega_b,c)+2\pi.
$$
\end{proof}

Let us further introduce the following notation: let
$$
\Lambda_k(\gamma) = \sup_{c}\Lambda_k(M_\gamma,c),
$$
denote the supremum of $\bar{\lambda}_k(M,g)$ over \emph{all metrics} on the closed, orientable surface  $M_\gamma$ of genus $\gamma$. Similarly, for Steklov eigenvalues we define
$$
\Sigma_k(\gamma,b) = \sup_c\Sigma_k(\Omega_{\gamma, b},c),
$$
where $\Omega_{\gamma,b}$ is an orientable surface of genus $\gamma$ with $b$ boundary components.

\begin{theorem}[Petrides~\cite{Petrides, PetridesSteklov}]
Assume that 
\begin{equation}
\label{Lambda1:condition}
\Lambda_1(\gamma)>\Lambda_1(\gamma-1),
\end{equation}
where $\Lambda_1(-1)$ is set to be $0$ by definition.
Then there is a metric $g$ on $M_\gamma$ such that $\bar\lambda_1(M_\gamma, g)=\Lambda_1(\gamma)$, induced by a branched minimal immersion, by first eigenfunctions, into some sphere $\mathbb{S}^n$.

Assume that 
\begin{equation}
\label{Sigma1:condition}
\Sigma_1(\gamma,b)>\max\{\Sigma_1(\gamma,b-1),\Sigma_1(\gamma-1,b+1)\}.
\end{equation}
Then there exists a metric $g$ on $\Omega_{\gamma.b}$ such that $\bar\sigma_1(\Omega_{\gamma,b},g) = \Sigma_1(\gamma,b)$, induced by a (branched) free boundary minimal immersion, by first Steklov eigenfunctions, into some Euclidean ball $B^n$.
\end{theorem}
\begin{remark}
\label{strict:remark}
Note that the non-strict versions of inequalities~\eqref{Lambda1:condition},~\eqref{Sigma1:condition} are always satisfied.
\end{remark}
Note that it follows from the following lower bound proved in~\cite{BBD}
\begin{equation}
\label{ref:lower_bound}
\Lambda_1(\gamma)\geq \frac{3}{4}\pi(\gamma-1)
\end{equation}
that the inequality~\eqref{Lambda1:condition} holds for infinitely many values of $\gamma$.


The following proposition also appears in~\cite[Corollary 1.6]{GLS}. 
\begin{proposition}
\label{Sigma_limit:prop}
For $k=1,2$ one has 
$$
\lim_{b\to\infty} \Sigma_k(\gamma,b) =\Lambda_k(\gamma).
$$
\end{proposition}
\begin{proof}
Theorem~\ref{GLS:thm} implies that for all $k\geqslant 0$ 
$$
\lim_{b\to\infty}\Sigma_k(\gamma,b)\geqslant \Lambda_k(\gamma).
$$
At the same time, for any conformal class $c$ on $\Omega_{\gamma,b}$ one can glue-in the holes to obtain a conformal class $\bar c$ on $M_\gamma$ such that $(\Omega_{\gamma,b},c)\subset (M_\gamma,\bar c)$. Then by Theorem~\ref{Steklov:thm} for any $b$ and $k=1,2$ one has
$$
\Sigma_k(\gamma,b) = \sup_c\Sigma_k(\Omega_{\gamma,b},c)\leqslant \sup_{\bar c}\Lambda_k(M_\gamma,\bar c) = \Lambda_k(\gamma).
$$
\end{proof}
\begin{theorem}
\label{fbms:prop}
Let $\gamma$ be such that the condition~\eqref{Lambda1:condition} is satisfied. Then
one has 
\begin{equation}
\label{SigmaLambda:ineq}
\Sigma_1(\gamma,b)<\Lambda_1(\gamma) 
\end{equation}
and there are infinitely many $b$ such that the inequality~\eqref{Sigma1:condition} holds. 
\end{theorem} 
\begin{proof}
The non-strict version of inequality~\eqref{SigmaLambda:ineq} follows from the proof of Proposition~\ref{Sigma_limit:prop}.

We start with the second statement. Let $\gamma$ be fixed. Combining~\eqref{Lambda1:condition} and Proposition~\ref{Sigma_limit:prop} one has
$$
\lim_{b\to\infty}\Sigma_1(\gamma,b)>\lim_{b\to\infty}\Sigma_1(\gamma-1,b).
$$
Therefore, for large $b$ one has
$$
\Sigma_1(\gamma,b)> \Sigma_1(\gamma-1,b+1).
$$
Hence, it only remains to establish that for infinitely many of these large $b$ one has 
$$
\Sigma_1(\gamma,b)>\Sigma_1(\gamma,b-1).
$$
Assume the contrary. Then by Remark~\ref{strict:remark} for large enough $b$ one has that
$$
\Lambda_1(\gamma) = \Sigma_1(\gamma, b),
$$
which would also violate the claimed strict inequality~\eqref{SigmaLambda:ineq}.
Then there exists $b_0$ such that 
$$
\Lambda_1(\gamma) = \Sigma_1(\gamma, b_0)>\Sigma_1(\gamma, b_0-1).
$$
We claim that $\Sigma_1(\gamma, b_0)>\Sigma_1(\gamma-1,b_0+1)$. Indeed, otherwise by Remark~\ref{strict:remark} one has equality $\Sigma_1(\gamma, b_0)=\Sigma_1(\gamma-1,b_0+1)$ and, thus, by Remark~\ref{strict:remark}
$$
\Lambda_1(\gamma-1) = \lim_{b\to\infty}\Sigma_1(\gamma-1,b)\geqslant \Sigma_1(\gamma-1,b_0+1) = \Sigma_1(\gamma, b_0) = \Lambda_1(\gamma),
$$
which contradicts~\eqref{Lambda1:condition}. As a result, one has that the condition~\eqref{Sigma1:condition} is satisfied for $(\gamma,b_0)$, i.e. there exists a metric $g$ on $\Omega_{\gamma,b_0}$ such that  
$$
\bar\sigma_1(\Omega_{\gamma,b_0},g) = \Sigma_1(\gamma,b_0) = \Lambda_1(\gamma).
$$
Let $(M_\gamma,\bar g)$ be obtained by gluing-in the holes in $(\Omega_{\gamma,b_0},g)$ so that $(\Omega_{\gamma,b_0},g)\subset (M_\gamma,\bar g)$. Then, by Theorem~\ref{Steklov:thm} one has
$$
\Lambda_1(\gamma) = \bar\sigma_1(\Omega_{\gamma,b_0},g) < \Lambda_1(M_\gamma,[\bar g])\leqslant\Lambda_1(\gamma),
$$
which is a contradiction.
\end{proof}

Finally, we recall the connection to free boundary minimal surfaces, see e.g.~\cite{FS, FS2}.

\begin{theorem}
There are infinitely many values of $\gamma\geqslant 0$ satisfying 
$$
\Lambda_1(\gamma)>\Lambda_1(\gamma-1),
$$
where $\Lambda_1(-1)$ is set to be $0$.
For each such $\gamma$ there are infinitely many $b\geqslant 1$ such that the value $\Sigma_1(\gamma,b)$ is achieved by a smooth metric. In particular, there exists a free boundary minimal branched immersion $f\colon \Omega_{\gamma,b}\to B^{n_{\gamma,b}}$ by the first Steklov eigenfunctions.
\end{theorem}


\begin{remark}
We expect that the results of this section extend to non-orientable surfaces. However, to the best of authors' knowledge, the analog of condition~\eqref{Sigma1:condition} for non-orientable surfaces or of lower bound~\eqref{ref:lower_bound} have not appeared in the literature, so we refrain from stating the non-orientable version of Theorem~\ref{fbms:prop} here. Note that the non-orientable analog of~\eqref{Lambda1:condition} can be found in~\cite{MS}.
\end{remark}

\subsection{Proof of Theorem~\ref{measures:thm}}\label{meas.thm.pf}

In light of the min-max characterization provided by Theorem \ref{MainTh2:intro}, Theorem \ref{measures:thm} is an immediate consequence of the following proposition.

\begin{proposition}\label{en.meas.bd} Let $\mu\in [C^0(M)]^*$ be an admissible probability measure on $M$, and fix a conformal class of metrics $c=[g]$ on $M$. Then
$$\lambda_1(\mu,c)\leq 2\mathcal{E}_n(M,c),$$
with equality if and only if
$$\lambda_1(M,c,\mu) \mu=|du|_g^2\,dv_g$$
for some harmonic map $u: (M,g)\to \mathbb{S}^n$ of spectral index $1$.

If in addition $\lambda_1(M,c,\mu)>0$, then   
$$\lambda_2(\mu,c)\leq 2 \mathcal{E}_{n,2}(M,c)$$
with equality only if
$$\lambda_2(M,c,\mu)\mu=|du|_g^2\,dv_g$$
for a harmonic map $u: (M,g)\to \mathbb{S}^n$ of spectral index $2$.
\end{proposition}

The proof follows roughly the same lines as that of Propositions \ref{lambda.below} and \ref{lambda2.below} for the volume measures, with some aid from the following technical lemma.

\begin{lemma}\label{wk.cvg.lem} Let $\mu$ be an admissible probability measure, with associated map $T: W^{1,2}(M,g)\to L^2(M,\mu)$. For any sequence $\varphi_j$ which is bounded in $W^{1,2}$ and converges weakly to $\varphi \in W^{1,2}$, we also have the convergence
$$T(\varphi_j)\to T(\varphi)\text{ in }L^2(\mu).$$
\end{lemma}
\begin{proof} Since $\varphi_j$ is bounded in $W^{1,2}$, it follows from definition of admissibility that, after passing to a subsequence, the functions $T(\varphi_j)$ converge strongly 
$$T(\varphi_j)\to \psi\text{ in }L^2(\mu).$$
Now, for any $\eta\in L^2(\mu)$, the continuity of $T$ implies that the linear functional
$$W^{1,2}\ni f\mapsto \langle T^*(\eta),f\rangle:=\int_M T(f) \eta d\mu$$
defines an element $T^*(\eta)\in (W^{1,2})^*$ of the dual space to $W^{1,2}$; thus, since $\varphi_j\to \varphi$ weakly in $W^{1,2}$, it follows that
\begin{eqnarray*}
\int_M T(\varphi)\eta&=&\langle T^*(\eta),\varphi\rangle\\
&=&\lim_{j\to\infty} \langle T^*(\eta),\varphi_j\rangle\\
&=&\lim_{j\to\infty}\int_MT(\varphi_j)\eta d\mu\\
&=&\int_M \psi \eta d\mu.
\end{eqnarray*}
It follows that $T(\varphi)=\psi$, as desired.
\end{proof}

%

The proof of Theorem~\ref{en.meas.bd} is now fairly straightforward.

\begin{proof}[Proof of Proposition \ref{en.meas.bd}]

By definition of $\mathcal{E}_n(M,g)$, we can find a sequence $\epsilon_j\to 0$ and a sequence of families $F^j\in \Gamma_n(M)$ such that
$$\lim_{j\to\infty}\max_{y\in B^{n+1}}E_{\epsilon_j}(F^j_y)=\mathcal{E}_n(M,c).$$
Since the map $T:W^{1,2}(M,g)\to L^2(\mu)$ is continuous, we see that the map
$$B^{n+1}\ni y \mapsto \int_M T(F_y^j)d\mu \in \mathbb{R}^{n+1}$$
is a continuous map coinciding with the identity $\mathbb{S}^n\to \mathbb{S}^n$ on the boundary $\partial B^{n+1}$. Thus, it follows as before that there exists $y_j\in B^{n+1}$ such that the maps $u_j=F^j_{y_j}$ satisfy
$$\int_M T(u_j) d\mu=0\in \mathbb{R}^{n+1},$$
while
\begin{equation}\label{gl.bd}
\limsup_{j\to\infty} \int_M|du_j|^2+\frac{1}{2\epsilon_j^2}(1-|u_j|^2)^2\leq 2\mathcal{E}_n(M,c).
\end{equation}

Passing to a subsequence, by Banach-Alaoglu, we can find a map $u\in W^{1,2}(M,\mathbb{R}^{n+1})$ such that
$$u_j\to u\text{ weakly in }W^{1,2}\text{ and strongly in }L^2(M).$$
By Lemma \ref{wk.cvg.lem}, it also follows that
$$T(u_j)\to T(u)\text{ strongly in }L^2(\mu),$$
and since $\int_M (1-|u_j|^2)^2=O(\epsilon_j^2)$, the limit map $u$ must satisfy
$$|u|\equiv 1\text{ in }L^2(M)$$
and
$$|T(u)|\equiv 1\text{ in }L^2(\mu).$$
Combining all this information, we see that
\begin{equation}
\lambda_1(\mu,c)\leq \int_M |du|^2\leq \liminf_{j\to\infty}\int_M |du_j|^2\leq 2\mathcal{E}_n(M,c),
\end{equation}
from which the desired estimate follows. 

In the case of equality $\lambda_1(\mu,c)=2\mathcal{E}_n(M,c)$, we see that each inequality in the chain above is an equality, from which it follows that
$$u_j\to u\text{ strongly in }W^{1,2}(M,g)$$
and the nonzero components $u^i\in W^{1,2}(M,\mathbb{R})$ of $u$ minimize the Rayleigh quotient among functions with $\mu$-average $0$. In particular, it follows that
\begin{equation}\label{eigen.cond}
\int_M\langle du,dv\rangle_g \,dv_g=\lambda_1(\mu,c)\int_M \langle T(u),T(v)\rangle d\mu
\end{equation}
for all $v\in W^{1,2}(M,\mathbb{R}^{n+1})$. 

Now, if $v\in W^{1,2}(M,\mathbb{R}^{n+1})\cap L^{\infty}$ satisfies $\langle u,v\rangle \equiv 0$ in $W^{1,2}(M,g)$, it's easy to see that 
$$\langle T(u),T(v)\rangle \equiv T(\langle u,v\rangle)\equiv 0$$
in $L^2(\mu)$ as well. As a consequence, for any map $v\in W^{1,2}(M,\mathbb{R}^{n+1})\cap L^{\infty}$, testing the map $w=v-\langle v,u\rangle u$ (which is pointwise perpendicular to $u$) in \eqref{eigen.cond} gives
$$\int_M\langle du,d(v-\langle v,u\rangle u\rangle)\rangle_g\,dv_g=\int_M\langle du,dv\rangle-\int_M|du|^2\langle v,u\rangle=0.$$
In particular, it follows that $u:(M,g)\to \mathbb{S}^n$ is weakly harmonic, and setting $v=\varphi u$, we see that
$$\lambda_1(\mu,c)\int_M \varphi d\mu=\int_M\langle du,dv\rangle=\int_M|du|^2\varphi \,dv_g,$$
so that
$$\lambda_1(\mu,c)\mu=|du|^2 \,dv_g,$$
as claimed.

The proof of the second inequality $\lambda_2(M,c,\mu)\leq 2\mathcal{E}_{n,2}(M,c)$ follows similar lines. Assume now that $\mu$ admits a ``first eigenfunction" $\phi \in W^{1,2}(M,g)$ minimizing the Rayleigh quotient $\int_M|d\phi|^2/\|T(\phi)\|_{L^2(\mu)}^2$ among all $\phi$ with $\int_M T(\phi)d\mu=0$, so that
$$\int_M \langle d\phi, d\psi\rangle_g \,dv_g=\lambda_1(\mu,c)\int_MT(\phi)T(\psi)d\mu$$
for all $\psi \in W^{1,2}(M,g)$.

As before, consider a sequence of families $F^j\in \Gamma_{n,2}(M)$ such that 
$$\lim_{j\to\infty}\max_{y\in [B^{n+1}]^2}E_{\epsilon_j}(F_y^j)=\mathcal{E}_{n,2}(M,c).$$
By appealing to Lemma \ref{degree.lem.2} and the continuity of the map $T: W^{1,2}(M,g)\to L^2(M,\mu)$, we see that the averaging maps
$$[\overline{B}^{n+1}]^2\ni (a,b)\mapsto \left(\int_M F_{a,b} d\mu, \int_M \phi F_{a,b}d\mu\right)\in \mathbb{R}^{2(n+1)}$$
must have a zero. In particular, we can extract from the families $F^j$ sequence of maps
$$u_j=F_{a_j,b_j}^j:M\to \mathbb{R}^{n+1}$$
such that
\begin{equation}\label{0mean.cond}
\int_Mu_jd\mu=\int_M\phi u_jd\mu=0\in \mathbb{R}^{n+1}
\end{equation}
while
$$\limsup_{j\to\infty}\int_M |du_j|^2+\frac{1}{2\epsilon_j^2}(1-|u_j|^2)^2\leq 2\mathcal{E}_{n,2}(M,c).$$

Once again, by appealing to Banach-Alaoglu and Lemma \ref{wk.cvg.lem}, we can pass to a subsequence to find a weak limit $u$ of the sequence $u_j$, such that
$$u_j\to u\text{ strongly in }L^2(M),\text{ }T(u_j)\to T(u)\text{ in }L^2(\mu),$$
and
$$|u|\equiv |T(u)|\equiv 1.$$
Now, it follows from \eqref{0mean.cond} that
$$\int_M ud\mu=\int_M\phi ud\mu=0\in \mathbb{R}^{n+1},$$
so we see that each nonzero component $u^i$ of $u$ satisfies
$$\int_M |du^i|^2\,dv_g\geq \lambda_2(\mu,c)\int_M (u^i)^2d\mu,$$
and summing over $i=1,\ldots,n+1$ gives
$$\int_M |du|^2\,dv_g\geq \lambda_2(\mu,c).$$
In particular, it follows that
$$\lambda_2(\mu,c)\leq \int_M |du|_g^2\,dv_g\leq \liminf_{j\to\infty}\int_M |du_j|_g^2\,dv_g\leq 2\mathcal{E}_{n,2}(M,c).$$
This gives the desired estimate for the case $k=2$, and the proof of the rigidity result in the case of equality follows exactly the same lines as the proof of the corresponding result in the case $k=1$.

\end{proof}

\end{document}